\documentclass[twoside,a4paper,12pt,centertags]{amsart}

\usepackage{amscd}
\usepackage{amsmath,amssymb,verbatim,vmargin,slashed,graphicx,subcaption}
\usepackage[all]{xy}
\usepackage[bookmarks=true]{hyperref}   % Hyperref in DVI and PDF
\usepackage{color}
\usepackage{tikz}
\usetikzlibrary{arrows,decorations.markings}

\theoremstyle{plain}
\newtheorem{thm}{Theorem}[section]

\newtheorem{prop}[thm]{Proposition}

\theoremstyle{definition}
\newtheorem{defn}[thm]{Definition}
\newtheorem{rem}[thm]{Remark}

\newtheorem{ex}[thm]{Example}

\mathchardef\semic="303B

\newcommand{\swp}{\text{{\rm WP}}\,}
\newcommand{\diag}{\text{{\rm diag}}\,}
\newcommand{\emm}{\text{{\rm em}}\,}
\newcommand{\bdy}{\partial}

\newcommand{\dirac}{{\mathbf D}}
\newcommand{\del}{\delta}
\newcommand{\wedg}{\mathbin{\scriptstyle{\wedge}}}
\newcommand{\clp}{\mathbin{\scriptscriptstyle{\triangle}}}
\newcommand{\lctr}{\mathbin{\lrcorner}}

\newcommand{\inv}[1]{{\widehat{#1}}}
\newcommand{\rev}[1]{\overline{#1}}
\newcommand{\R}{{\mathbf R}}

\newcommand{\C}{{\mathbf C}}

\newcommand{\mH}{{\mathcal H}}

\DeclareMathOperator{\re}{Re}
\newcommand{\im}{\text{{\rm Im}}\,}
\newcommand{\sett}[2]{ \{ #1 \, \semic \, #2 \} }

\newcommand{\scl}[2]{\langle #1,#2 \rangle}

\newcommand{\clos}[1]{\overline{#1}}

\newcommand{\barint}{\mbox{$ave \int$}}
\newcommand{\divv}{{\text{{\rm div}}}}
\newcommand{\curl}{{\text{{\rm curl}}}}

\newcommand{\ta}{{\scriptstyle\top}}
\newcommand{\no}{{\scriptstyle\perp}}
\newcommand{\pd}{\partial}

\newcommand{\pv}{\text{{\rm p.v.\!}}}

\newcommand{\nindex}{{\overline n}}

\def\barint_#1{\mathchoice
            {\mathop{\vrule width 6pt
height 3 pt depth -2.5pt
                    \kern -8.8pt
\intop}\nolimits_{#1}}%
            {\mathop{\vrule width 5pt height
3 pt depth -2.6pt
                    \kern -6.5pt
\intop}\nolimits_{#1}}%
            {\mathop{\vrule width 5pt height
3 pt depth -2.6pt
                    \kern -6pt
\intop}\nolimits_{#1}}%
            {\mathop{\vrule width 5pt height
3 pt depth -2.6pt
          \kern -6pt \intop}\nolimits_{#1}}}

\usepackage{color}

\definecolor{gr}{rgb}   {0.,   0.8,   0. }
\definecolor{bl}{rgb}   {0.,   0.5,   1. }
\definecolor{mg}{rgb}   {0.7,  0.,    0.7}

\makeatletter
\@namedef{subjclassname@2010}{%
  \textup{2010} Mathematics Subject Classification}
\makeatother

\begin{document}

\title[Dirac integral equations for dielectric and plasmonic scattering]
{Dirac integral equations for dielectric and plasmonic scattering}

\author[Johan Helsing]{Johan Helsing$\,^1$}
\author[Andreas Ros\'en]{Andreas Ros\'en$\,^2$}
\thanks{$^1\,$Centre for Mathematical Sciences, Lund University,
         Box 118, 221 00 Lund, Sweden. ORCID: 0000-0003-1434-9222}
\thanks{$^2\,$Mathematical Sciences, Chalmers University of Technology and University of Gothenburg, 
SE-412 96 G{\"o}teborg, Sweden. ORCID: 0000-0003-2393-1804.
Corresponding author: andreas.rosen@chalmers.se}

%\address{Andreas Ros\'en\\Mathematical Sciences, Chalmers University of Technology and University of Gothenburg\\
%SE-412 96 G{\"o}teborg, Sweden.}
%\email{andreas.rosen@chalmers.se}

%\address{Johan Helsing\\Centre for Mathematical Sciences, Lund University,
%         Box 118, 221 00 Lund, Sweden} 
%\email{johan.helsing@math.lth.se}

%\date{\today}

\begin{abstract}
  A new integral equation formulation is presented for the Maxwell
  transmission problem in Lipschitz domains. It builds on the Cauchy
  integral for the Dirac equation, is free from false eigenwavenumbers
  for  a wider range of permittivities than other known formulations, 
  can be used for magnetic
  materials, is applicable in both two and three dimensions, and does
  not suffer from any low-frequency breakdown. Numerical results for
  the two-dimensional version of the formulation, including examples
  featuring surface plasmon waves, demonstrate competitiveness
  relative to state-of-the-art integral formulations that are
  constrained to two dimensions.   However, our Dirac integral equation
  performs equally well in three dimensions, as demonstrated in
  a companion paper. 
\end{abstract}

\keywords{Maxwell scattering, Boundary integral equation, Spurious resonances,
Clifford-Cauchy integral, Surface plasmon wave, Non-smooth object, Nystr\"om discretization}
\subjclass[2010]{45E05, 78M15, 15A66, 65R20}

\maketitle
%\today

%
%
%
\section{Introduction}

This paper introduces a new boundary integral equation (BIE)
formulation for solving the time-harmonic Maxwell transmission problem
across interfaces between domains of constant and isotropic, but
otherwise general complex-valued, permittivities and permeabilities.
 The main novelty of our BIE is that it performs well for a very wide range
of materials, including metamaterials with negative permittivities.
We provide not only numerical evidence of such performance,
   but also rigorous proofs for general Lipschitz interfaces. 

We refer to our new formulation as a Dirac integral equation, since
 our point of departure  is to embed Maxwell's equations into a
Dirac-type equation by relaxing the constraints $\divv E=0$ and $\divv
B=0$ on the electric and magnetic fields $E$ and $B$ and introducing
two auxiliary Dirac variables, which make the partial differential
equation (PDE) elliptic also without these divergence-free
constraints. This means that our integral equation uses eight unknown
scalar densities in three dimensions (3D) and four unknown scalar
densities in two dimensions (2D), where eight and four, respectively,
are the dimensions of the Clifford algebra which is the algebraic
structure behind the Dirac equation and which we make use of in the analysis. 
 This idea to write Maxwell's equations as a Dirac equation
is old, and appears in the works of M. Riesz and D. Hestenes.
Even Maxwell himself formulated his equations using Hamilton's quaternions. 
Picard~\cite{PicardDir:84,PicardDir:85} elliptifies Maxwell's equations 
by embedding these into a Dirac equation, 
without using 
Clifford algebra but writing it as a div-curl-grad matrix system.  
This Picard's extended Maxwell system is the basis for more recent 
works \cite{TaskinenYlaOijala:06,TaskinenVanska:07,SchultzHiptmair:20}.

 In harmonic analysis of non-smooth boundary value problems for 
Maxwell's equations, BIEs based on the Dirac
equation were developed 
by M. Mitrea, McIntosh et al., see 
\cite {McIntoshMitrea:99, AxGrognardHoganMcIntosh:00}.
Building on this, the second author of the present paper further developed
the theory of Dirac integral equations for solving Maxwell's equations
 in his PhD work \cite{AxPhD:03,AxThesisPub1:04,AxThesisPub2:03,AxMcIntosh:04, AxThesisPub4:06} with
Alan McIntosh.
Applications of these methods to scattering from perfect electric conductors (PEC) are found in \cite{RosenSpin:17,RosenBoosting:19}, and the spin integral equations 
there are precursors of the Dirac integral equations presented here.
More recent results on Dirac equations for Maxwell scattering problems with
Lipschitz interfaces are also \cite{MarmMitreaShi:12,HernHerr:19}, which deal with
the $L_p$ boundary topology, but only treat the case of equal wave numbers
in the two domains.

The focus in the present paper is on BIEs which are numerically good
for a very wide range of material combinations, including such where
permittivity ratios are negative real numbers. In this case, the
relevant function space is the trace space for the $L_2$ topology in
the domains, since this is where surface plasmons waves may appear, as
we approach purely negative permittivity ratios. In particular, we seek
BIEs without false eigenwavenumbers in this plasmonic regime. In the
simpler cases of scattering from PEC objects or from dielectric
objects with less permissible restrictions on their permittivities,
and also on the genus of the domains, there are many good formulations
available. See, for example,~\cite{GaneshHawkinsVolkov:14} for a BIE
with six scalar unknowns containing only weakly singular integral
operators and extending results from the pioneering
work~\cite{TaskinenYlaOijala:06}. However, for metamaterials in
plasmonic scattering, we are only aware of two BIEs without false
eigenwavenumbers, namely the Dirac integral equations presented here
and the BIE in~\cite{HelsingKarlsson:20}. Both these works seem to
indicate that eight is the number of scalar densities needed to
construct a Maxwell BIE in 3D that is free from false eigenwavenumbers
for all complex-valued choices of material constants.  

Turning to the description of the present work, the main objective is to design the
integral equations so that not only the operator has good Fredholm
properties, but also so that no false eigenwavenumbers appear. This
problem may be best explained by the relation
\begin{displaymath}
  \text{PDE}\circ\text{Ansatz}=\text{BIE}\,.
\end{displaymath}
The PDE problem, Helmholtz or Maxwell, is to invert a map
$(F^+,F^-)\mapsto g$, where $g$ is boundary datum, and $F^\pm$ are the
solutions, the sought fields in the interior domain $\Omega^+$ and
exterior domain $\Omega^-$, respectively. To obtain a BIE we need to
choose an appropriate ansatz, that is an integral representation
$h\mapsto (F^+,F^-)$ of the fields in terms of a density function $h$
on the interface $\bdy\Omega$ between $\Omega^+$ and
$\Omega^-$. From this we obtain a BIE as the composed map $h\mapsto g$,
via $(F^+,F^-)$. A main objective is to use an invertible ansatz with
as good condition number as possible, so that the BIE is invertible
whenever the PDE problem is well posed.

Our main results, the Dirac integral equations for solving Maxwell's
equations in $\R^2$ as well as in $\R^3$, are formulated in
Section~\ref{sec:matrixform} and make use of ansatzes which are
invertible for all choices of material constants. At a more technical
level we make the following comments, where $k_\pm$ are wavenumbers in
$\Omega^\pm$, $\im k_\pm\ge 0$, and $\hat k=k_+/k_-$.

\begin{itemize}
\item Propositions~\ref{prop:helminj} and \ref{prop:maxwinj} show
  sufficient conditions for the Helmholtz and Maxwell transmission
  problems to have  uniqueness of  solutions. For non-magnetic materials, our
  result is the hexagonal region shown in Figure~\ref{fig:hexagon},
  which strictly contains the regions earlier obtained for the
  Helmholtz transmission problem in~\cite{KresRoac78,KleiMart88}. See,
  however, the proof of~\cite[Prop.~3.1]{HelsingKarlsson:20} for a comment on a minor flaw
  in~\cite{KresRoac78}.
\begin{figure}[h!]
\centering
\includegraphics[height=58mm]{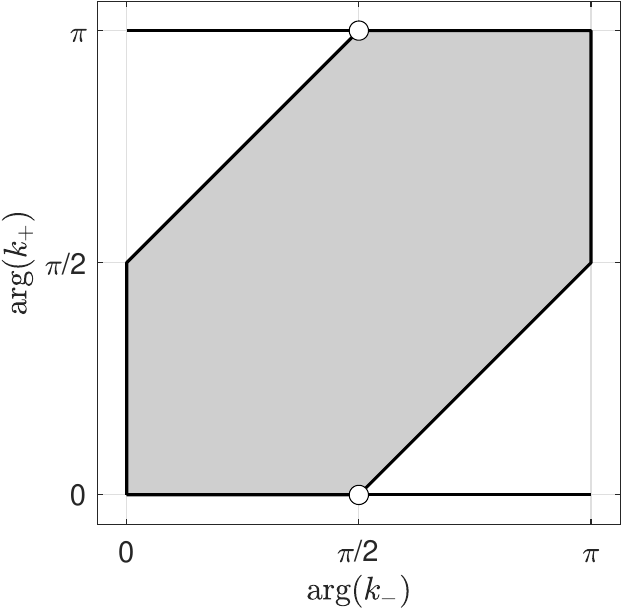}
\caption{\sf  The hexagonal region of $(\arg(k_-),\arg(k_+))$ where
the uniqueness condition from  Theorem~\ref{prop:helminj} holds,
in the case of non-magnetic materials.
For such, the condition reads $\hat k\in \swp(k_-,k_+)$,
with notation from Definition~\ref{defn:hexagon}. }
\label{fig:hexagon}
\end{figure}
When also the PDE problem defines a Fredholm map, something that can
fail in the hexagon only when $|\arg(k_+)-\arg(k_-)|=\pi/2$, then we
also have existence of solutions. Outside this region of existence and
uniqueness, and for a given $k_-$, the uniqueness of solutions fails
only for a discrete set of $k_+$.
  
\item The Dirac BIE, presented in Section~\ref{sec:matrixform}, is
  only one possible choice from a family of Dirac BIEs derived in
  Sections~\ref{sec:diracint2d} and \ref{sec:3DDirac}. In $\R^2$,
  these BIEs depend on four Dirac parameters $r,\beta,\alpha',\beta'$.
  The choice for $r$ is critical for numerical performance, whereas
  the precise choices for the other parameters seem to be of lesser
  importance as long as they are chosen so that no false
  eigenwavenumbers appear. In $\R^3$, there are two additional
  parameters $\gamma,\gamma'$.
  
  The Dirac BIE is constructed in Propositions~\ref{prop:diracintwp}
  and \ref{prop:diracintwp3d} using a dual PDE problem as ansatz. Our
  choice of dual PDE problem giving the Dirac BIEs in
  Section~\ref{sec:matrixform}, has to effect that these BIEs are
  invertible whenever the PDE problem is well posed.
  
\item Propositions~\ref{prop:helmexist} and \ref{prop:maxwexist} show
  that the PDE problem defines a Fredholm map provided that the
  quotient $\hat\epsilon$ of the permittivities is not in an interval
  $[-C, -1/C]$ for some $C\ge 1$ depending on the Lipschitz regularity
  of the interface. These estimates use Hodge potentials for the
  fields and concern the physical energy norm, corresponding to a
  boundary function space, $\mH_2$ or $\mH_3$, which roughly speaking
  is a suitable mix
  of the function space $H^{1/2}(\bdy\Omega)\subset L_2(\bdy\Omega)$
  of traces to Sobolev $H^1(\Omega)$ functions and its dual space
  $H^{-1/2}(\bdy\Omega)\supset L_2(\bdy\Omega)$.
   In the case of pure metamaterials $\hat\epsilon<0$, the PDE problem
  to be solved may fail to be a Fredholm map, and hence no 
  integral equation for solving it can be Fredholm. 
  However, when the PDE problem has unique solution we can
  solve it by an injective Dirac integral equation, and when Fredholmness 
  fails we solve the PDE problem numerically for smooth boundary data by
  taking a limit from $\im\hat\epsilon>0$.  
  Numerical experiments indicate the existence of this limit.
  Theoretically it should be possible to prove its existence 
  by analyzing the problem in
  slightly larger fractional Sobolev spaces. 
  
  More precisely, in  $\R^n$, $n\ge 2$,  one can show that the PDE problem
   \eqref{eq:Helmtransmprobl}  defines
  a Fredholm map when $(\hat\epsilon+1)/(\hat\epsilon-1)$ is
  outside the essential spectrum of the double layer potential
  operator
  \begin{equation}  
     K_{\rm d} f(x):=\int_{\bdy\Omega}\scl{\nu(y)}{(\nabla\Phi)(y-x)}f(y)d\sigma(y),
     \qquad x\in\bdy\Omega,
  \label{eq:doublelayer}
  \end{equation}
  where $\Phi$ is the Laplace fundamental solution and $\nu$ is the
  outward unit normal. 
   Another basic operator for the  Maxwell problem \eqref{eq:maxwtransprclassi}    
  in $\R^3$ is the magnetic dipole
  operator 
  \begin{equation}
     K_{\rm m} f(x):=
         \nu(x)\times\int_{\bdy\Omega}(\nabla\Phi)(y-x)\times f(y)d\sigma(y),
     \qquad x\in\bdy\Omega,
  \end{equation}
  acting on tangential vector fields $f$. 
  Compact perturbations of the operators $K_{\rm d}$ and $K_{\rm m}$,
  and their adjoints $K_{\rm d}^*$ and $K_{\rm m}^*$
   with respect to the standard $L_2$ pairing on $\bdy \Omega$, appear along the diagonals of our basic Cauchy
  integral operators \eqref{eq:Ek2D} and \eqref{eq:Ek3D}.  In
  particular, the (3:4,3:4) and (7:8,7:8) size
  $2\times 2$ diagonal blocks in \eqref{eq:Ek3D} are compact perturbations
  of $-K_{\rm m}^*$. 
   
  It is important to note that in the energy norm, the essential
  spectra of $K_{\rm d}$ and $K_{\rm m}$ are subsets of $(-1,1)$. If
  considering instead the $L_2(\bdy \Omega)$ norm, then on domains in $\R^2$
  with one corner, the essential spectrum of $K_{\rm d}$ is a lying
  ``figure eight'', centered around $0$. See Fabes, Jodeit and
  Lewis~\cite{FabesJodeitLewis:77} and the final comments and
  \eqref{eq:eight} in the present paper.
  
\item Although our main concern is wavenumbers $k_\pm\ne 0$, it is
  important to have a good behaviour of the BIEs as frequency
  $\omega\to 0$. Typical problems which can occur are dense-mesh
  low-frequency breakdown and topological low-frequency breakdown. We
  refer to \cite{VicGreFer18,EpsteinGreengardONeil:19} for a more
  detailed discussion and for BIEs that are immune to such
  low-frequency breakdown, and remark that this immunity is shared by
  the Dirac BIEs presented in Section~\ref{sec:matrixform}.
  One reason
  for this is that our four densities include the gradient of the
  field in $\R^2$ and our eight densities include the electric and the
  magnetic fields in $\R^3$, so that no numerical differentiation is
  needed. Moreover, if $k_\pm\to 0$ with $\hat k$ constant in
  \eqref{eq:2DDiracequ} or \eqref{eq:3DDiracequ}, then $P,P',N,N'$ are
  all constant whereas the Cauchy singular integral operators
  $E_{k_\pm}$ converge to $E_0$ in operator norm. 
   That the limit Dirac BIE does not have any false eigenwavenumbers,
  for fixed $0<|\hat k|<\infty$ as $k_\pm\to 0$, is shown in 
  \cite[Sec. 7]{HelsingKarlssonRosen:20}. 
  
   We remark that BIEs derived from the Picard system
   \cite{TaskinenYlaOijala:06,TaskinenVanska:07} do not suffer from
   dense-mesh low-frequency breakdown, but are known to have
   false/near-false eigenwavenumbers. See \cite[p.~163]{VicGreFer18}
   for an excellent survey of such ``charge-current formulations''. 
   Note that the Picard equation~\cite[eq.~(45)]{TaskinenVanska:07} is
   the Dirac equation~\eqref{eq:dirac} written in matrix form, and that 
   the BIEs \cite[eqs.~(43),(44)]{TaskinenVanska:07} are $8\times 8$ systems
   of the second kind with single and double potential type operator blocks. 
   Thus there is much formal resemblance with the Dirac BIE, but the
   real difference lies in the precise choice of parameters made in
   the present paper. 
\end{itemize}

The paper ends with Section~\ref{sec:numerical}, where the numerical
performance of the Dirac integral equations is studied in $\R^2$.
Among other things, the performance is compared to that of a
state-of-the-art system of integral equations of direct (Green's
theorem method) type~\cite{KleiMart88,HelsingKarlsson:20} which is
only applicable in $\R^2$, which involves only two unknown scalar
densities, and which for certain $k_\pm$ coincides with a 2D version
of the classic M{\"u}ller system\cite[p. 319]{Muller69}. Not
surprisingly this special-purpose system performs best, but the new
$4\times 4$ Dirac system is not far behind.
In particular it performs well under computationally challenging plasmonic conditions,
that is when $\hat\epsilon$ is negative and real.

 In the companion paper \cite{HelsingKarlssonRosen:20} by the authors 
joint with A. Karlsson,  we test the new $8\times 8$ Dirac system for
the Maxwell transmission problem in $\R^3$ using the numerical
techniques in~\cite{HelsingKarlsson:16}. All BIE systems that we are
aware of in the literature, with size less than $8\times 8$, exhibit
false eigenwavenumbers -- primarily at the left middle corner point in
Figure~\ref{fig:hexagon}. 
Our Theorem~\ref{thm:3Ddiracint} gives very weak sufficient conditions for our Dirac integral
equation \eqref{eq:3DDiracequ} not to have false eigenwavenumbers. When these conditions fail,
notably at the bottom middle corner point in
Figure~\ref{fig:hexagon}, 
it is still possible to avoid false
eigenwavenumbers with other choices of Dirac parameters in 
Section~\ref{sec:3DDirac}.

\section{Matrix Dirac integrals}   \label{sec:matrixform}

This section states, in classical vector and matrix notation, the
integral equations which we propose for solving the Maxwell
transmission problems. The derivation of these equations in later
sections makes use of multivector algebra. We consider a bounded
interior domain $\Omega^+$, separated by the interface $\bdy \Omega$
from the exterior domain $\Omega^-:= \R^n\setminus \clos{\Omega^+}$. Fix
a large $R<\infty$ and define
$\Omega^-_R:=\sett{x\in\Omega^-}{|x|<R}$. For simplicity we assume
throughout this paper that $\Omega^-$ is connected, although many
results do not need this assumption.

The parameters that we use for problem description in the transmission
problems \eqref{eq:Helmtransmprobl} and \eqref{eq:maxwtransprclassi},
that we aim to solve, are wavenumbers $k_+$ and $k_-$ in $\Omega^+$
and $\Omega^-$ respectively, and a jump parameter $\hat \epsilon$. We
assume that $\im k_\pm\ge 0$, but unless otherwise stated we do not
assume any relation between $k_+, k_-$ and $\hat \epsilon$. We denote
the ratio between the wavenumbers by $\hat k= k_+/k_-$. Our main
interest is in scattering for non-magnetic materials where
$\hat\epsilon= \hat k^2$. In applications, the parameter $\hat
\epsilon$ appears as the ratio $\hat \epsilon= \epsilon_+/\epsilon_-$
of the permittivities $\epsilon_\pm$ in $\Omega^\pm$. Similarly for
magnetic materials, we have a ratio $\hat \mu= \mu_+/\mu_-$ of the
permeabilities $\mu_\pm$ in $\Omega^\pm$. In this general case, we
have the relation $\hat\epsilon= \hat k^2/\hat \mu$.

We remark that in what follows, $E^\pm$ denotes the standard electric
fields, but we have rescaled the magnetic fields $B^\pm$ so that what
we here call $B^\pm$, in standard notation reads
$B^\pm/\sqrt{\epsilon_\pm\mu_\pm}$.

To formulate our results, we need the following conical subsets of $\C$.
We use the argument $-\pi<\arg(z)\le \pi$.
\begin{defn}  \label{defn:hexagon}
  Let $k_-, k_+\in\C\setminus\{0\}$ be such that $\im k_-\ge 0$ and $\im k_+\ge 0$.
  Define $\phi_\pm:= |\arg (k_\pm/i)|$.
  Let $\swp(k_-,k_+)\subset\C\setminus\{0\}$ be the set of
  $z\in\C\setminus\{0\}$ satisfying
\begin{equation}
\begin{cases}
  |\arg(z)|\le \pi-\phi_+-\phi_-,& \text{if } \phi_+<\pi/2,\phi_++\phi_->0,\\
  |\arg(z)|<\pi, &  \text{if } \phi_+=\phi_-=0,\\
  \min(|\arg(z)|,|\arg(-z)|)\le \pi/2-\phi_-, & \text{if } \phi_+=\pi/2, 0<\phi_-\le \pi/2,\\
    \re z\ne 0, & \text{if } \phi_+=\pi/2, \phi_-=0.
\end{cases}
\end{equation}
\end{defn}  
  
Our notation is to denote fields in the domains by $F, U,\ldots$, with
suitable superscript $\pm$, and to write $f,u,\ldots$ for respective
boundary traces. On $\bdy \Omega$ we denote by $\nu$ the unit normal
vector pointing into the exterior domain $\Omega^-$. Furthermore
$\{\nu,\tau\}$ and $\{\nu,\tau,\theta\}$ denote positive ON-frames on
$\bdy\Omega$ depending on dimension, so that $\tau$ is the
counter-clockwise tangent on curves $\bdy\Omega\subset\R^2$. On
surfaces $\bdy\Omega\subset\R^3$, the theory we develop works for any
choice of tangent ON-frame $\{\tau,\theta\}$. We denote the
directional derivative in direction $v$ by $\pd_v$, and with slight
abuse of notation , $\pd_\nu u$ denotes normal derivative on
$\bdy\Omega$ although this use values of $U$ in a neighbourhood of
$\bdy\Omega$.
  
Our main result for 2D scattering concerns the Helmholtz transmission
problem
\begin{equation}   \label{eq:Helmtransmprobl}
\begin{cases}
   u^+= u^-+u^0, & x\in\bdy \Omega, \\
   \pd_\nu u^+={\hat\epsilon}\pd_\nu (u^-+ u^0), & x\in\bdy \Omega,\\
    \Delta U^+ +k_+^2 U^+=0, &x\in\Omega^+,\\
   \Delta U^- +k_-^2 U^-=0, &x\in\Omega^-,\\
    \pd_{x/|x|}U^- -ik_- U^-=o(|x|^{-(n-1)/2}e^{\im k_- |x|}), & x\to\infty,
   \end{cases}
\end{equation}
where $\Omega^+\subset\R^2$, $n=2$, and with $u^0\in
H^{1/2}(\bdy\Omega)$ being the trace of an incoming wave $U^0$, and we
want to solve for $U^+\in H^1(\Omega^+)$ and $U^-\in H^1(\Omega^-_R)$.
 We remark that for $\im k_->0$, this not so well-known sufficient 
radiation condition 
of exponential growth entails a necessary condition of 
exponential decay. See \cite[Prop. 9.3.6]{RosenGMA:19} for proofs.

Let $\diag D$ denote the square diagonal matrix with diagonal $D$,
and let $A^{\rm T}$ denote the transpose of a matrix $A$.  
Define
\begin{align}  
  P &= \diag\begin{bmatrix}(\hat k+|\hat k|)^{-1/2} & (\hat k+|\hat k|)^{-1/2} & ({\hat\epsilon}+1)^{-1} & 1 \end{bmatrix}, \label{eq:2DPN1} \\
  P' &= \diag\begin{bmatrix}(\hat k+|\hat k|)^{-1/2} & (\hat k+|\hat k|)^{-1/2} &  1 & |\hat k|(\hat k+|\hat k|)^{-1}\end{bmatrix},\\
  N &= \diag\begin{bmatrix} \hat k(\hat k+|\hat k|)^{-1/2} & |\hat k|(\hat k+|\hat k|)^{-1/2} & {\hat\epsilon}({\hat\epsilon}+1)^{-1} & 1 \end{bmatrix},\\
   N' &= \diag\begin{bmatrix}|\hat k|(\hat k+|\hat k|)^{-1/2} & \hat k(\hat k+|\hat k|)^{-1/2} &  1 & \hat k(\hat k+|\hat k|)^{-1}\end{bmatrix}. \label{eq:2DPN4} 
\end{align} 

\begin{thm}   \label{thm:2Ddiracint}
  Let $\Omega^+\subset\R^2$ be a bounded Lipschitz domain, with
  $\Omega^-$ being connected and notation as above. Then there exists
  a constant $1\le C(\bdy\Omega)<\infty$ depending on the Lipschitz
  constants of the parametrizations of $\Omega^+$ by smooth domains,
  so the following holds.
  
  The transmission problem \eqref{eq:Helmtransmprobl} is well posed if
\begin{equation}  \label{eq:wpcond2D}
  {\hat\epsilon}\in \C\setminus [-C(\bdy \Omega), -1/C(\bdy \Omega)] \quad\text{and}\quad {\hat\epsilon} /\hat k\in \swp(k_-,k_+).
\end{equation}

For its solution, consider the Dirac integral equation
\begin{equation}  \label{eq:2DDiracequ}
  (I+PE_{k_+}N'-NE_{k_-}P')h= 2Nf^0
\end{equation}
for four scalar functions  $h=[h_1\;h_2\;h_3\;h_4]^{\rm T}$,  where $E_k$ is
the singular integral operator \eqref{eq:Ek2D} which we introduce in
Section~\ref{sec:cauchy}, $P,P', N, N'$ are the constant diagonal
matrices \eqref{eq:2DPN1}--\eqref{eq:2DPN4}, and 
\begin{equation}   \label{eq:helm2drhs}
f^0= \begin{bmatrix} ik_-u^0 & 0 & \pd_\nu u_0 & \pd_\tau u_0 \end{bmatrix}^{\rm T}. 
\end{equation} 
The operator in \eqref{eq:2DDiracequ} is invertible on the energy
trace space
$\mH_2$ from \eqref{eq:H2space}, introduced in Section~\ref{sec:diracint2d},
 whenever \eqref{eq:wpcond2D} holds and $\hat k\in \C\setminus(-\infty,0]$.
Moreover, the solution to \eqref{eq:Helmtransmprobl} in $\Omega^\pm$ is
 obtained
from $h^+= N' h$ and $h^-= P'h$ as
\begin{equation}   \label{eq:ufield2d}
 U^\pm= \tfrac 1{2i k_\pm}\begin{bmatrix} -\tilde K^{\nu'}_{k_\pm} &
 -\tilde K^{\tau'}_{k_\pm} & \tilde S^1_{k_\pm} & 0 \end{bmatrix} h^\pm
\end{equation}
and 
\begin{equation}   \label{eq:ugradfield2d}
 \nabla U^\pm= \tfrac 12\begin{bmatrix} \tilde S^{\nu'}_{k_\pm} &
  \tilde S^{\tau'}_{k_\pm} &
 -\tilde K^I_{k_\pm} & - \tilde K^J_{k_\pm} \end{bmatrix} h^\pm,
\end{equation} 
with notation as in Section~\ref{sec:cauchy}, 
where $\tilde K$ and $\tilde S$ denote layer potentials,
$\nu'$ and $\tau'$ denote the normal and tangential vector at the point
of integration $y\in\bdy \Omega$,
and $I=\begin{bmatrix} 1 & 0 \\ 0 & 1
\end{bmatrix}$ and $J=\begin{bmatrix} 0 & -1 \\ 1 & 0 \end{bmatrix}$. 
\end{thm}

\begin{proof}
  This result is derived in Sections~\ref{sec:helmwp} and
  \ref{sec:diracint2d}, where we arrive at equation
  \eqref{eq:2DfinalDirac} with $\alpha={\hat\epsilon}$. We
  precondition \eqref{eq:2DfinalDirac} by multiplying from the left by
  $P$ and writing $\tilde h= P' h$ to obtain \eqref{eq:2DDiracequ}
  with $N=PM$, $N'=\hat kM'P'$ and $P(\hat kM'+M)P'=I$. We remark that
  the parameter ratios $\hat k<0$ and ${\hat\epsilon}=-1$ must be
  excluded for $P$ and $P'$ to be well-defined. However, the only
  problem when $\hat k<0$ is that we cannot precondition
  \eqref{eq:2DfinalDirac} with $P$ and $P'$ to achieve a non-integral
  term $I$ in \eqref{eq:2DDiracequ}.
  
  In applications of the Dirac problem \eqref{eq:diractranspr} to the
  Helmholtz problem \eqref{eq:Helmtransmprobl}, as in
  Example~\ref{ex:TMMaxw}, we have $f^0= ik_- u^0 +\nabla u^0$, that
  is \eqref{eq:helm2drhs}. Equations (\ref{eq:ufield2d}) and
  (\ref{eq:ugradfield2d}) for the fields are seen from the first  row  and the
  last two rows in \eqref{eq:Ek2D} respectively, and the expressions
  for $h^\pm$ are seen from the right  and middle factors  in \eqref{eq:factors2d}.
  Note that $h^-$ happens to coincide with the density $\tilde h$,
  which is introduced in \eqref{eq:gendirint} in a different context.
\end{proof}

We remark that it is our experience that the precise choice of $P$ and $P'$
for Theorem~\ref{thm:2Ddiracint}, as well as for Theorem~\ref{thm:3Ddiracint},
is less important as long as they satisfy $P(\hat kM'+M)P'=I$.
With the choices made, our aim has been to minimize the growth of
$P$, $P'$, $N$ and $N'$ as we vary $|k|$.

We next formulate the analogous result for the Maxwell transmission
problem in $\R^3$. Our transmission problem is
\begin{equation}  \label{eq:maxwtransprclassi}
\begin{cases}
  \nu\times E^+= \nu\times (E^-+ E^0), & x\in \bdy\Omega,\\
   \nu\times B^+=(\hat k/\hat \epsilon) \nu\times (B^-+B^0), & x\in \bdy\Omega,\\
   \nabla\times E^+= ik_+ B^+, \nabla\times  B^+= -ik_+ E^+, &x\in\Omega^+,\\
   \nabla\times E^-= ik_- B^-, \nabla\times B^-= -ik_- E^-, &x\in\Omega^-,\\
   x/|x|\times E^- - B^-= o(|x|^{-1}e^{\im k_- |x|}), & x\to\infty,\\
      x/|x|\times B^- - E^-= o(|x|^{-1}e^{\im k_- |x|}), & x\to\infty,\\
   \end{cases}
\end{equation} 
where $E^0$ and $B^0$ are the incoming electric and magnetic fields,
and we want to solve for $E^\pm$ and $B^\pm$. All fields are assumed
to be $L_2$ integrable in a neighbourhood of $\bdy \Omega$.

Define matrices (where we write $\hat c=1/\hat k$ and $\hat \mu= \hat
k^2/\hat\epsilon$)
\begin{align}  
  P &= \diag\begin{bmatrix} \frac{1}{{\hat c}+1} & \frac 1{\sqrt{{\hat c}+|{\hat c}|}}
  & \frac {1}{{\hat\mu}+1}\frac 1{\sqrt {\hat c}} & \frac {1}{{\hat\mu}+1}\frac 1{\sqrt {\hat c}} & \frac {|{\hat c}|}{{\hat c}+ |{\hat c}|} & \frac {\hat \epsilon}{\hat\epsilon+1} & 1 & 1     \end{bmatrix}, \label{eq:3DPN1} \\
  P' &= \diag\begin{bmatrix} 1 & \frac 1{\sqrt{{\hat c}+|{\hat c}|}}
  & \frac 1{\sqrt {\hat c}} & \frac 1{\sqrt {\hat c}} & 1 & 1 & \frac 1{{\hat c}+ 1} & \frac 1{{\hat c}+ 1} 
     \end{bmatrix},  \\
  N &= \diag\begin{bmatrix} \frac{{\hat c}}{{\hat c}+1} & \frac {\hat c}{\sqrt{{\hat c}+|{\hat c}|}}
  & \frac {{\hat\mu}}{{\hat\mu}+1}{\sqrt {\hat c}} &  \frac {{\hat\mu}}{{\hat\mu}+1}{\sqrt {\hat c}} & \frac {{\hat c}}{{\hat c}+ |{\hat c}|}  & \frac {1}{{\hat\epsilon}+1} & 1 & 1    \end{bmatrix},  \\
  N' &= \diag\begin{bmatrix} 1 & \frac {|{\hat c}|}{\sqrt{{\hat c}+|{\hat c}|}}
  & \sqrt {\hat c} & \sqrt {\hat c} & 1  & 1 & \frac {\hat c}{{\hat c}+ 1} & \frac {\hat c}{{\hat c}+ 1} 
    \end{bmatrix}. \label{eq:3DPN4} 
\end{align}

\begin{thm}   \label{thm:3Ddiracint}
  Let $\Omega^+\subset\R^3$ be a bounded Lipschitz domain, with
  $\Omega^-$ being connected and notation as above. Then there exists
  a constant $1\le C(\bdy\Omega)<\infty$ depending on the Lipschitz
  constants of the parametrizations of $\Omega^+$ by smooth domains,
  so the following holds.
  
  The transmission problem \eqref{eq:maxwtransprclassi} is well posed if
\begin{equation}  \label{eq:wpcond3D}
  {\hat\epsilon}\in \C\setminus [-C(\bdy \Omega), -1/C(\bdy \Omega)] \quad\text{and}\quad {\hat\epsilon} /\hat k\in \swp(k_-,k_+).
\end{equation}

For its solution, consider the Dirac integral equation
\begin{equation}  \label{eq:3DDiracequ}
  (I+PE_{k_+}N'-NE_{k_-}P')h= 2Nf^0
\end{equation}
for eight scalar functions  $h=[h_1\; h_2\; h_3\; h_4\; h_5\; h_6\;
h_7\; h_8]^{\rm T}$,  where $E_k$ is the singular integral operator
\eqref{eq:Ek3D} which we introduce in Section~\ref{sec:cauchy}, $P,P',
N, N'$ are the constant diagonal matrices
\eqref{eq:3DPN1}--\eqref{eq:3DPN4}, and 
\begin{equation}  \label{eq:maxwrhs}
f^0= \begin{bmatrix}  0 & B^0_\nu & B^0_\tau & B^0_\theta & 0
& E^0_\nu & E^0_\tau & E^0_\theta  \end{bmatrix}^{\rm T},
\end{equation} 
with field components in the frame $\{\nu,\tau,\theta\}$. The operator
in \eqref{eq:3DDiracequ} is invertible on the energy trace space
$\mH_3$ from \eqref{eq:H3space}, introduced in
Section~\ref{sec:3DDirac}, whenever \eqref{eq:wpcond3D}, ${\hat
  k}^2/\hat\epsilon\in \C\setminus [-C(\bdy \Omega), -1/C(\bdy
\Omega)]$ and $\hat k\in \C\setminus(-\infty,0]$ holds. Moreover, the
solution to \eqref{eq:maxwtransprclassi} in $\Omega^\pm$ is obtained
from $h^+= N' h$ and $h^-= P'h$ as
\begin{equation}   \label{eq:Bfield3d}
 B^\pm=  \tfrac 12\begin{bmatrix} \tilde K_{k_\pm}^{\nu'\times\cdot} & -\tilde K_{k_\pm}^{I} & -\tilde K_{k_\pm}^{\theta'\times\cdot} & \tilde K_{k_\pm}^{\tau'\times\cdot} & \tilde S_{k_\pm}^{\nu'} & 0 & \tilde S_{k_\pm}^{\theta'} &-  \tilde S_{k_\pm}^{\tau'}  \end{bmatrix} h^\pm
\end{equation}
and 
\begin{equation} \label{eq:Efield3d}
E^\pm= \tfrac 12 \begin{bmatrix} \tilde S_{k_\pm}^{\nu'} & 0 & - \tilde S_{k_\pm}^{\theta'} &  \tilde S_{k_\pm}^{\tau'} &
- \tilde K_{k_\pm}^{\nu'\times\cdot} & -\tilde K_{k_\pm}^{I} & -\tilde K_{k_\pm}^{\theta'\times\cdot} & \tilde K_{k_\pm}^{\tau'\times\cdot} \end{bmatrix} h^\pm,
\end{equation} 
with notation as in Section~\ref{sec:cauchy}, 
where $\tilde K$ and $\tilde S$ denote layer potentials,
$\nu',\tau',\theta'$ denote the frame vectors at the point
of integration $y\in\bdy \Omega$, $I=\begin{bmatrix} 1 & 0 \\ 0 & 1
\end{bmatrix}$, and $v\times\cdot$ denotes the map
$x\mapsto v\times x$. 
\end{thm}

\begin{proof}
  This result is derived in Section~\ref{sec:3DDirac}, where we arrive
  at equation \eqref{eq:3DfinalDirac} with $\alpha=\hat\epsilon$. We
  precondition \eqref{eq:3DfinalDirac} by multiplying from the left by
  $P$ and writing $\tilde h= P' h$ to obtain \eqref{eq:3DDiracequ}
  with $N=PM$,  $N'=\hat k^{-1}M'P'$  and $P(\hat k^{-1}M'+M)P'=I$. We remark that
  the parameter ratios $\hat k<0$, $\hat\mu= {\hat
    k}^2/\hat\epsilon=-1$ and ${\hat\epsilon}=-1$ must be excluded for
  $P$ and $P'$ to be well-defined. However, the only problem when
  $\hat k<0$ is that we cannot precondition \eqref{eq:3DfinalDirac}
  with $P$ and $P'$ to achieve a non-integral term $I$ in
  \eqref{eq:3DDiracequ}.
  
  In applications of the Dirac problem \eqref{eq:diractranspr} to the
  Maxwell problem \eqref{eq:maxwtransprclassi}, or in multivector
  notation \eqref{eq:maxwtranspr}, we have $F^0= E^0 +B^0$, that is
  \eqref{eq:maxwrhs}. Equations (\ref{eq:Bfield3d}) and
  (\ref{eq:Efield3d}) for the fields are seen from  the second to fourth rows and last
  three rows  in \eqref{eq:Ek3D} respectively.
\end{proof}

\section{Multivector algebra}

This section  contains the basics of multivector
algebra which we need.  
For a complete account of the theory of
multivector algebra and Dirac equations which we use, we refer to
\cite{RosenGMA:19}.
 For our purposes,  multivectors are the same  objects  as Cartan's alternating
forms, and multivector fields are  the  same as differential forms, and
amounts to an algebra of not only the one-dimensional vectors but
also $j$-dimensional algebraic objects for $0\le j\le n$ in $\R^n$.
Concretely, multivectors in $\R^n$ are the following type of
objects, where our interest is in $n=2,3$.
Denote by $\{e_1,\ldots, e_n\}$ the standard vector basis for $\R^n$.
We write $\wedge\R^n$ for the complex $2^n$ dimensional space of
multivectors in $\R^n$, which is spanned by basis multivectors $e_s$,
where $s\subset\nindex= \{1,2,\ldots, n\}$. We write $\wedge^j\R^n$ for the
subspace spanned by  those $e_s$ with $j$ number of elements in the
index set $s$, so that
\begin{equation}
  \wedge\R^n= \wedge^0\R^n \oplus \wedge^1\R^n\oplus 
  \wedge^2\R^n\oplus\cdots \oplus \wedge^n\R^n.
\end{equation} 
We identify  $\wedge^0\R^n=\C$ and $\wedge^1\R^n=\C^n$ with 
the scalars and vectors respectively.
 Objects in $\wedge^2 \R^n$ are referred to as bivectors. 

In $\R^2$, a multivector is of the form
\begin{equation}
  w= a+ v_1e_1+v_2e_2 + be_{12},
\end{equation}
and so amounts to two scalars $a,b$ and a vector $v= v_1e_1+v_2e_2$.
In $\R^3$, a multivector is of the form
\begin{equation}
  w= a+ v_1e_1+v_2e_2+v_3e_3+ 
  u_1e_{23}+ u_2e_{13}+ u_3 e_{12}+ be_{123},
\end{equation}
and so amounts to two scalars $a,b$ and two vectors $v= v_1e_1+v_2e_2+v_3e_3$
and
$u= u_1e_1+u_2e_2+u_3e_3$ via Hodge duality  (see below). 

On this $2^n$ dimensional space $\wedge\R^n$ we use three products
which, depending on dimension, generalize the classical vector operations:  the exterior product $u\wedg w$,
the (left) interior product $u\lctr w$ and the Clifford product $uw$.
 The last is the standard short hand notation for the Clifford product, whereas
the notation $e_s\clp e_t$ was introduced in \cite{RosenGMA:19}, for reasons
which are clear from \eqref{eq:defclp}. 
 Products of basis multivectors $e_s$
correspond to signed unions, set differences and symmetric differences of index sets
$s$ respectively. 
More precisely, for index sets $s,t\subset \nindex$ we define the permutation sign
$\epsilon(s,t)= (-1)^{|\sett{(s_i,t_j)\in s\times t}{s_i>t_j}|}$,
that is the sign of the permutation that rearranges $s\cup t$ in increasing order.
Then 
\begin{align}
  e_s\wedg e_t &=
  \begin{cases} 
  \epsilon(s,t) e_{s\cup t}, & s\cap t=\emptyset,\\
  0, & s\cap t\ne \emptyset,
  \end{cases}\\
  e_s\lctr e_t &=
  \begin{cases} 
  0, & s\not\subseteq t,\\
  \epsilon(s,t\setminus s) e_{t\setminus s}, & s\subseteq t,
  \end{cases}\\
  e_se_t &=   \epsilon(s,t) e_{s\triangle t},  \label{eq:defclp}
\end{align}
 where $s\triangle t= (s\cup t)\setminus (s\cap t)$ is the symmetric difference of sets. 
The exterior and Clifford products are associative, but the interior product is
not.
On the one hand, orthogonal vectors in $\wedge^1\R^n$ anticommute both with respect to the 
exterior and Clifford products. 
On the other hand,  $e_j\wedg e_j=0$ whereas $e_je_j=1$, $j=1,\ldots, n$. 
An important special case of the interior product is the Hodge star duality
\begin{equation}
   {*w}= w\lctr e_\nindex.
\end{equation}
The most important instance of the above three products is when the first factor
is a vector, and not a general multivector. In this case, the three products are related
as
\begin{equation}  
   uw= u\lctr w+u\wedg w,\qquad u\in \wedge^1 \R^n, w\in \wedge \R^n.
\end{equation} 
A multivector $w\in \wedge \R^n$ at a point $x\in \bdy \Omega$, with
normal vector $\nu\in \wedge^1 \R^n$, is called {\em tangential}
if $\nu\lctr w=0$ and is called {\em normal} if $\nu\wedg w=0$. 
 
 \begin{ex}[$\R^2$ multivector algebra]   \label{ex:R2prods}
A general multivector in $\R^2$ is of the form
$w= a_1+v+*a_2$, where $a_1,a_2\in\wedge^0\R^2$ and $v\in\wedge^1\R^2$.
The three basic multivector products with a vector $u\in\wedge^1\R^2$ can be written 
\begin{align}
  u\wedg w &= 0 + a_1u+ *((*u)\cdot v),\\
  u\lctr w &= u\cdot v + a_2 (*u)+ 0,\\
  u w&= u\cdot v+(a_1 u+a_2 (*u))+ *((*u)\cdot v).
\end{align} 
Note that on vectors $u\in\wedge^1\R^2$, the Hodge star $*u$
is counter clockwise rotation $\pi/2$. 
\end{ex} 

\begin{ex}[$\R^3$ multivector algebra]   \label{ex:R3prods}
A general multivector in $\R^3$ is of the form
$w= a_1+v_1+*v_2+*a_2$, where $a_1,a_2\in\wedge^0\R^3$ and $v_1,v_2\in\wedge^1\R^3$.
The three basic multivector products with a vector $u\in\wedge^1\R^3$ can be written
\begin{align}
  u\wedg w &= 0 + ua_1+ *(u\times v_1)+ *(u\cdot v_2),\\
  u\lctr w &= u\cdot v_1 - u\times v_2+ *(u a_2)+ 0,\\
  u w&= u\cdot v_1+(ua_1-u\times v_2)+ *(u\times v_1+u a_2)+ *(u\cdot v_2).
\end{align}
\end{ex}
We refer to \cite[Chapters 2,3]{RosenGMA:19} for further details of 
multivector algebra.

In the same way that the scalar and vector products induce the 
divergence $\nabla\cdot F$ and curl $\nabla\times F$ in vector calculus,
the exterior, interior and Clifford products induce the exterior, 
 interior and Clifford/Dirac derivatives in multivector calculus,
 and we write these as
 \begin{align}
    dF(x)=\nabla\wedg F(x) &= \sum_{j=1}^n e_j \wedg \pd_{e_j} F(x),\\
    \del F(x)= \nabla\lctr F(x)  &= \sum_{j=1}^n e_j \lctr \pd_{e_j} F(x),\\
    \dirac F(x)= \nabla\clp F(x) &= \sum_{j=1}^n e_j \pd_{e_j} F(x).
 \end{align}
Replacing $u$ and $w$ by $\nabla$ and $F(x)$ in Example~\ref{ex:R3prods},
it follows how these differential operators can be written in terms of 
divergence, curl and gradient in $\R^3$.
We refer to \cite[Chapters 7,8]{RosenGMA:19} for further details of multivector 
calculus. 

The time-harmonic wave Dirac equation which we consider in this paper is
\begin{equation}  \label{eq:dirac}
  \dirac F(x)= ik F(x),
\end{equation}
for multivector fields $F$, and wave number $k\in\C$. 
The main applications are the Helmholtz and Maxwell
equations. 
Given a scalar solution $U$ to the Helmholtz equation $\Delta U+ k^2 U=0$, we define
the multivector field 
\begin{equation}   \label{eq:helmgrad}
F_{\text{helm}}= \nabla U+ ik U.
\end{equation}  
It follows that $F=F_{\text{helm}}$ satisfies \eqref{eq:dirac} since
$\dirac^2=\Delta$.

For Maxwell's equations,  define the total electromagnetic (EM) field to be the
multivector field
\begin{equation}   \label{eq:totalEM}
  F_\emm(x)= E_1(x)e_1+E_2(x)e_2+E_3(x)e_3+
  B_1(x)e_{23}  - B_2(x)e_{13}  + B_3(x) e_{12},
\end{equation} 
where $E$ is the electric field and $B$ is the magnetic field, and the
$\wedge^0\R^3$ and $\wedge^3\R^3$ parts are zero. In this formalism,
the Dirac equation \eqref{eq:dirac} for $F=F_\emm$ given by
\eqref{eq:totalEM} coincides with the time-harmonic Maxwell's
equations,  that is the PDE from \eqref{eq:maxwtransprclassi}.  Indeed, the $\wedge^0\R^3$, $\wedge^1\R^3$, $\wedge^2\R^3$
and $\wedge^3\R^3$ parts of \eqref{eq:dirac} are the Gauss,
Maxwell--Amp\`ere, Faraday and magnetic Gauss law respectively. See
\cite[Sec. 9.2]{RosenGMA:19}.

\section{The $\R^n$ Cauchy integral}  \label{sec:cauchy}

A main reason for using the Dirac framework is that it provides us
with a Cauchy-type reproducing formula, which allows for a
generalization of complex function theory to $n\ge 2$ and $k\ne 0$.
See \cite[Chapters 8,9]{RosenGMA:19} for further details. More
precisely, if $F$ satisfies \eqref{eq:dirac} in a domain $\Omega$ with
boundary $\bdy\Omega$, then a Cauchy-type reproducing formula
\begin{equation}  
  F(x)=\int_{\bdy\Omega}\Psi_k (y-x)\nu(y)f(y)d\sigma(y),\qquad x\in\Omega,
\label{eq:fieldCauchy}
\end{equation}
holds.
We write $dy$ and $\nu\in\wedge^1\R^n$ for the standard measure and outward pointing unit normal on $\bdy \Omega$ respectively,
and the integrand uses two Clifford products.
The first factor
\begin{equation}  \label{eq:dirfundsol}
  \Psi_k(x)= -\tfrac 12(\nabla \Phi_k(x)-ik\Phi_k(x))\in \wedge^1\R^n\oplus \wedge^0\R^n\subset\wedge\R^n
\end{equation}
is a fundamental solution to the elliptic operator $\dirac+ik$,
as $\dirac \Psi_k(x)+ik\Psi_k(x)= \delta(x)$.
Here $\Phi_k(x)\in\wedge^0\R^n=\C$ is the Helmholtz fundamental solution, and we use the normalization 
from \cite{ColtonKress:83} so that $(\Delta+k^2)\Phi_k= -2\delta(x)$. Hence the factor $-1/2$ in \eqref{eq:dirfundsol}.
In dimension $n=3$ we have 
\begin{equation}   \label{eq:phik3d}
  \Phi_k(x)= (2\pi|x|)^{-1}e^{ik|x|},
\end{equation}
and in dimension $n=2$ we have in terms of the 
Hankel function $H^{(1)}_0(z)$ that
\begin{equation}  \label{eq:phik2d}
  \Phi_k(x)= (i/2) H^{(1)}_0(k|x|).
\end{equation}

The classical theory of complex Hardy spaces generalizes from complex
function theory to our Dirac setting. Our basic operator, acting on a
suitable space $\mH$ of functions $h: \bdy \Omega\to\wedge\R^n$ on
$\bdy \Omega$, is the Cauchy principal value integral
\begin{equation}   \label{eq:cliffcauchy}
  E_k h(x)= 2\pv \int_{\bdy \Omega} \Psi_k (y-x)\nu(y) h(y) d\sigma(y),\qquad x\in\bdy \Omega,
\end{equation}
which reduces to the classical Cauchy integral when $n=2$, $k=0$.
The basic operator algebra is that $E_k^2=I$, and 
\begin{equation}   \label{eq:Ekpm}
  E_k^\pm =(I\pm E_k)/2
\end{equation}
are two complementary projection operators, that is $(E_k^\pm)^2=
E_k^\pm$. The operator $E_k^\pm$ projects onto its range, the subspace
of $\mH$ which we denote $E_k^\pm\mH$ and consists of all traces
$F|_{\bdy \Omega}$ of fields satisfying $\dirac F= ikF$ in
$\Omega^\pm$. For the exterior domain $\Omega^-$ these fields also
satisfy a Dirac radiation condition; See \eqref{eq:diractranspr} below
for its formulation.

For computations, we express the Cauchy singular operator $E_k$ as a
matrix with entries being double and single layer potential-type
operators, using the notation
\begin{align}
  K^v_k h(x) &= \pv\int_{\bdy\Omega} \scl{v(x,y)}{(\nabla \Phi_k)(y-x)} 
                h(y) d\sigma(y), \qquad x\in\bdy\Omega,\\
  S^a_k h(x) &= ik\int_{\bdy\Omega} a(x,y) \Phi_k(y-x) h(y) d\sigma(y),
  \qquad x\in\bdy\Omega,
\end{align}
respectively. Here $v(x,y)$ is a vector field and $a(x,y)$ is a scalar
function, depending on $x,y\in \bdy\Omega$. The single layer $S^a_k$
is a weakly singular integral operator, and so is also $K^v_k$ if
$\bdy\Omega$ is smooth and $v=\nu$ is the normal direction at $x$ or
$y$. Otherwise $K^v_k$ is a principal value singular integral, but is
bounded on many natural function spaces $\mH$, also for general
Lipschitz interfaces $\bdy\Omega$.
 
Consider first dimension $n=2$, with a curve $\bdy\Omega$ and $\Phi_k$
given by \eqref{eq:phik2d}. Here we use a positively oriented frame
$\{\nu,\tau\}$ at $x\in\bdy \Omega$, with $\nu=\nu(x)$ being the
normal vector into $\Omega^-$ and $\tau=\tau(x)$ the tangential vector
counter clock-wise from $\nu$. The corresponding frame at $y\in
\bdy\Omega$ we write as $\nu'=\nu(y)$ and $\tau'=\tau(y)$. In the
plane, the Clifford algebra $\wedge\R^2$ is four dimensional and
spanned by $\{1, e_1, e_2, e_{12}\}$. We prefer the ordering $\{1,
e_{12}, e_1, e_2\}$, since Clifford multiplication by vectors then
will be represented by block off-diagonal matrices. At $x\in\bdy
\Omega$, we write 
\begin{equation}   \label{eq:2dframe}
  h= h_1+h_2 \nu\tau+ h_3\nu+ h_4\tau\approx 
  \begin{bmatrix} h_1 & h_2 & h_3 & h_4 \end{bmatrix}^{\rm T},
\end{equation} 
using instead the vector frame $\{\nu,\tau\}$. Here $\nu\tau=e_{12}\in\wedge^2\R^2$
does not depend on $x$, even though $\nu$ and $\tau$ do so.
By writing out $\nabla\Phi_k$ and the Clifford products in \eqref{eq:cliffcauchy},
we obtain
\begin{equation}  \label{eq:Ek2D}
  E_k=
  \begin{bmatrix}
    -K^{\nu'}_k & -K^{\tau'}_k & S^1_k & 0 \\
    K^{\tau'}_k & -K^{\nu'}_k & 0 & S^1_k \\
    S^{\nu\cdot\nu'}_k & S^{\nu\cdot\tau'}_k & -K^\nu_k & K^\tau_k \\
     S^{\tau\cdot\nu'}_k &  S^{\tau\cdot\tau'}_k & - K^\tau_k & -K^\nu_k
  \end{bmatrix},
\end{equation}
in the multivector frame $\{1,\nu\tau, \nu,\tau\}$.

Next consider dimension $n=3$, with a surface $\bdy\Omega$ and
$\Phi_k$ given by \eqref{eq:phik3d}. Here we use a positively oriented
ON-frame $\{\nu,\tau,\theta\}$ at $x\in\bdy \Omega$, with $\nu=\nu(x)$
being the normal vector into $\Omega^-$ and $\tau=\tau(x)$ and
$\theta=\theta(x)$ being tangential vector fields. The frame vectors
at $y\in \bdy\Omega$, we write as $\nu'=\nu(y)$, $\tau'=\tau(y)$ and
$\theta'=\theta(y)$. A multivector field at $x\in\bdy \Omega$ we write
as
\begin{equation}   \label{eq:3dframe}
  h= h_1+h_2\tau\theta + h_3\theta\nu+ h_4\nu\tau+h_5\nu\tau\theta+h_6\nu+h_7\tau+h_8\theta.
\end{equation}
Here $\nu\tau\theta=e_{123}\in\wedge^3\R^3$ does not depend on $x$, although
in general each of the three vectors do so.
Again, we prefer this ordering of the frame multivectors since Clifford multiplication
by vectors
then is block off-diagonal.
In the multivector frame $\{1,\tau\theta,\theta\nu,\nu\tau,\nu\tau\theta,\nu,\tau,\theta\}$, 
we have 
\begin{equation} \label{eq:Ek3D}
E_k=
\begin{bmatrix}
  -K^{\nu'}_k  & 0 & K^{\theta'}_k &  -K^{\tau'}_k & 0 & S^1_k & 0 & 0  \\
  K^{\nu\times \nu'}_k & -K^\nu_k & -K^{\nu\times \theta'}_k & K^{\nu\times \tau'}_k & S^{\nu\cdot\nu'}_k &
  0 & S^{\nu\cdot \theta'}_k & -S^{\nu\cdot \tau'}_k  \\
  K^{\tau\times \nu'}_k & -K^\tau_k & -K^{\tau\times \theta'}_k & K^{\tau\times \tau'}_k & S^{\tau\cdot \nu'}_k &
  0 & S^{\tau\cdot\theta'}_k & -S^{\tau\cdot\tau'}_k  \\
  K^{\theta\times \nu'}_k & -K^\theta_k  & -K^{\theta\times \theta'}_k &  K^{\theta\times \tau'}_k & S^{\theta\cdot \nu'}_k & 0 & S^{\theta\cdot\theta'}_k & -S^{\theta\cdot \tau'}_k  \\
   0 & S^1_k  & 0 & 0 & -K^{\nu'}_k & 0 & -K^{\theta'}_k & K^{\tau'}_k \\
  S^{\nu\cdot\nu'}_k & 0 & -S^{\nu\cdot\theta'}_k & S^{\nu\cdot \tau'}_k & -K^{\nu\times\nu'}_k & 
  -K^\nu_k & -K^{\nu\times \theta'}_k & K^{\nu\times \tau'}_k  \\
  S^{\tau\cdot\nu'}_k & 0  & -S^{\tau\cdot\theta'}_k & S^{\tau\cdot \tau'}_k & -K^{\tau\times\nu'}_k & 
  -K^\tau_k & -K^{\tau\times \theta'}_k & K^{\tau\times \tau'}_k  \\
  S^{\theta\cdot\nu'}_k & 0  & -S^{\theta\cdot\theta'}_k & S^{\theta\cdot \tau'}_k & -K^{\theta\times\nu'}_k & 
  -K^\theta_k & -K^{\theta\times \theta'}_k & K^{\theta\times \tau'}_k  
\end{bmatrix}.
\end{equation}

Finally we note that we similarly can write the Cauchy integral
\eqref{eq:fieldCauchy} for the fields in $\Omega^\pm$ , in matrix
form. The only difference is the normalization factor $2$ and the fact
that we choose the frame $\{e_1,\ldots, e_n\}$ at the field point
$x\in \Omega^\pm$. We also allow for a general function $h:\bdy
\Omega\to \wedge\R^n$ and not only a trace $f$ of a solution to the
Dirac equation in \eqref{eq:fieldCauchy}. By the associativity of the
Clifford product, this still yields a field $F$ solving the Dirac
equation, but $h\ne F|_{\bdy \Omega}$ in general. In this case, when
$x\notin\bdy\Omega$, we denote the layer potentials by $\tilde K^v_k$
and $\tilde S^a_k$, with $v$ and $a$ now only depending on $y\in\bdy
\Omega$.  The field evaluation formulas 
\eqref{eq:ugradfield2d}, \eqref{eq:Bfield3d} and \eqref{eq:Efield3d} 
also use vector versions 
\begin{align} 
  \tilde K^A_k h(x) &= \int_{\bdy \Omega} A(y)(\nabla \Phi_k)(y-x) h(y) d\sigma(y), 
                \qquad x\notin\bdy\Omega,  
   \label{eq:fieldlayers1}\\
  \tilde S^v_k h(x) &= ik\int_{\bdy \Omega} v(y) \Phi_k(y-x) h(y) d\sigma(y),
                \qquad x\notin\bdy\Omega,  
\label{eq:fieldlayers2}
\end{align} 
of $\tilde K^v_k$ and $\tilde S^a_k$, where now 
$A(y):\R^n\to\R^n$ is a matrix function and $v(y)$ is a vector
field.

\section{The energy trace space}   \label{sec:tracespace}

We define in this section the appropriate norms and function spaces 
for the Dirac equation \eqref{eq:dirac}.
Let $\Omega^+\subset \R^n$ be a bounded Lipschitz domain.
The general Dirac transmission problem that we consider reads
\begin{equation}  \label{eq:diractranspr}
\begin{cases}
   f^+= M(f^-+f^0), & x\in\bdy \Omega,\\
   \dirac F^+= ik_+ F^+, &x\in\Omega^+,\\
   \dirac F^-= ik_- F^-, &x\in\Omega^-,\\
   (x/|x|-1)F^-= o(|x|^{-(n-1)/2}e^{\im k_- |x|}), & x\to\infty.
   \end{cases}
\end{equation}
This is our master transmission problem, into which we embed the
Helmholtz and Maxwell transmission problems in Sections~\ref{sec:diracint2d}
and \ref{sec:3DDirac}, where the multiplier $M$ will be specified.
The radiation condition stated here for the Dirac equation reduces to the 
classical Sommerfeld and Silver--M\"uller conditions 
for Helmholtz and Maxwell solutions $F_{\text{helm}}$ and $F_\emm$ respectively. 
See \cite[Sec. 9.3]{RosenGMA:19}.

For $F=F^+$ in $\Omega^+$, we use the norm
\begin{equation}
  \left( \int_{\Omega^+} (|F(x)|^2+ |\nabla\wedg F(x)|^2) dx \right)^{1/2}.
\end{equation}
It is important to note that for both Helmholtz and Maxwell fields 
$F_{\text{helm}}$ and $F_\emm$,
the second term $\nabla\wedg F$ is bounded by the first term $F$, and the norm reduces to the
$L_2$ norm of $F$.
In $\Omega^-$ we use the corresponding norm, but integrate only
over the bounded set $\Omega^-_R$.
Different choices of $R$ yield equivalent norms for solutions to the 
Dirac equation which satisfy the radiation condition.
See \cite[Sec. 9.3]{RosenGMA:19}.

It was shown in \cite{AxThesisPub4:06} 
that there exists a unique Hilbert space $\mH$ 
(although it does not come with a canonical norm) of 
multivector fields on the Lipschitz interface $\partial\Omega$, which is the 
trace space corresponding to the above norms on $F^\pm$ in $\Omega^\pm$,
with the following properties.
It splits into closed subspaces in two ways as 
\begin{equation}
  \mH= E_k^+\mH\oplus E_k^-\mH
  \quad\text{and}\quad\mH = \mH_\ta\oplus \mH_\no,
\end{equation}  
where $\ta$ means tangential and $\no$ means normal. 
There is a one-to-one correspondence $F^+\leftrightarrow f^+$ between
Dirac solutions $F^+$ in $\Omega^+$ and their traces $f^+\in E_k^+\mH$,
with inverse given by the Cauchy integral \eqref{eq:fieldCauchy}.
Similarly, there is a one-to-one correspondence $F^-\leftrightarrow f^-$ between
Dirac solutions $F^-$ in $\Omega^-$ satisfying the radiation condition,
and their traces $f^-\in E_k^-\mH$,
with inverse given by the Cauchy integral.

The subspace $\mH_\ta$ denotes the subspace of tangential multivector fields,
and there is a bounded and surjective tangential trace map
\begin{equation}
F\mapsto \nu\lctr(\nu\wedg f)\in \mH_\ta,
\end{equation}
of multivector fields $F$ in a neighbourhood $U$ of $\partial\Omega$ 
with norm $(\int_U (|F(x)|^2+ |\nabla\wedg F(x)|^2) dx )^{1/2}<\infty$.
The subspace $\mH_\no$ denotes the subspace of normal multivector fields,
and there is a bounded and surjective normal trace map
\begin{equation}
F\mapsto \nu\wedg(\nu\lctr f)\in \mH_\no,
\end{equation}
of multivector fields $F$ in a neighbourhood $U$ of $\partial\Omega$ 
with norm $(\int_U (|F(x)|^2+ |\nabla\lctr F(x)|^2) dx )^{1/2}<\infty$.

For smooth interfaces $\partial\Omega$, the trace space $\mH$ consists
of all $f\in H^{-1/2}(\partial\Omega;\wedge\R^n)$ such that 
$d_\ta f_\ta\in H^{-1/2}(\partial\Omega;\wedge\R^n)$
and $\del_\ta (\nu\lctr f)\in H^{-1/2}(\partial\Omega;\wedge\R^n)$,
with $d_\ta$ and $\del_\ta$ denoting the tangential exterior and interior 
derivatives along $\partial\Omega$.
Here $f_\ta$ denotes the tangential part of $f$ and $\nu\lctr f$ is the normal
part of $f$ with the factor $\nu$ removed.

To characterize $\mH$ in terms of fractional Sobolev-type spaces on 
non-smooth Lipschitz interfaces $\partial\Omega$ is a non-trivial problem.
Such a characterization was achieved for polyhedra in \cite{BuffaCiarlet1:01, BuffaCiarlet2:01}, for general Lipschitz interfaces in \cite{BuffaCostabelSheen:02} in $\R^3$, and extended to general 
multivector traces in \cite{Weck:04}.
For the energy trace space space $\mH$ on Lipschitz interfaces, this
shows the following.
For 2D domains, using the frame \eqref{eq:2dframe}, our
boundary function space is
\begin{equation}  \label{eq:H2space}
  \mH_2=\mH= H^{1/2}(\bdy\Omega)\oplus H^{1/2}(\bdy\Omega)\oplus H^{-1/2}(\bdy\Omega)\oplus H^{-1/2}(\bdy\Omega).
\end{equation}
It is for 3D domains that the non-trivial result from \cite[Thm. 4.1]{BuffaCostabelSheen:02}
is needed, which shows that in the frame \eqref{eq:3dframe}, we have
\begin{multline}   \label{eq:H3space}
  \mH_3 =\mH = 
  H^{1/2}(\bdy\Omega)\oplus
  H^{-1/2}(\bdy\Omega)\oplus
  *H^{-1/2}(\curl,\partial\Omega) \\
  \qquad\oplus 
   H^{1/2}(\bdy\Omega)\oplus
  H^{-1/2}(\bdy\Omega)\oplus
 H^{-1/2}(\curl,\bdy\Omega).
\end{multline}
Here $*$ is the 3D Hodge star and the space of tangential vector fields
is defined as 
\begin{equation}
H^{-1/2}(\curl,\bdy\Omega) =\sett{f\in (\nu\times H^{1/2}(\bdy\Omega)^3)^*}
{\curl_\ta f\in H^{-1/2}(\bdy\Omega)},
\end{equation}
where $\curl_\ta$ denotes tangential surface 
curl on $\bdy\Omega$ and  $(\nu\times H^{1/2}(\bdy\Omega)^3)^*$
denotes the dual space of $\nu\times H^{1/2}(\bdy\Omega)^3$.  
One should note that the test functions $\nu\times H^{1/2}(\bdy\Omega)^3$ will not be $H^{1/2}$
smooth on general Lipschitz regular $\bdy\Omega$, since $\nu$ is in general only
measurable.
To summarize, the conditions on $h_3,h_4,h_7,h_8$ are that 
$h_3\tau+h_4\theta\in H^{-1/2}(\curl,\bdy\Omega)$
and $h_7\tau+h_8\theta\in H^{-1/2}(\curl,\bdy\Omega)$.

 We remark that there is no canonical norm on $\mH_2$ or
$\mH_3$, but this does not present a problem and it should be clear from the
context in the
estimates to come, which choice among equivalent norms is used when 
such needs to be fixed.

 Throughout this paper $X\lesssim Y$ means that there exists $C<\infty$
independent of relevant variables so that $X\le CY$, and $X\approx Y$
means that $X\lesssim Y$ and $Y\lesssim X$.  

\section{Helmholtz existence and uniqueness}  
\label{sec:helmwp}

This section surveys basic solvability results for the Helmholtz
transmission problem \eqref{eq:Helmtransmprobl}, which are valid in
any dimension $n\ge 2$. The proofs follow \cite{AxThesisPub4:06}, with
a translation from the Dirac to the Helmholtz framework. Consider
first uniqueness of solutions $U^\pm$.

\begin{prop}   
\label{prop:helminj}
    Let $\Omega^+\subset \R^n$ be a bounded Lipschitz domain
    with connected exterior $\Omega^-$, and consider a solution $U^\pm$ to \eqref{eq:Helmtransmprobl}. 
Assume that the incoming wave
vanishes so that $u_0=0$. If  $\hat\epsilon$ from  \eqref{eq:Helmtransmprobl}
satisfies 
\begin{equation}   \label{eq:injcondforeps}
{\hat\epsilon} /\hat k\in \swp(k_-,k_+),
\end{equation} 
recalling that $\hat k= k_+/k_-$ and Definition~\ref{defn:hexagon}, 
then $U^+=U^-=0$ identically.
\end{prop}

\begin{proof}
From the jump relations we have
\begin{equation}
  \int_{\bdy \Omega} u^+ \rev{\pd_\nu u^+} d\sigma= \rev {\hat\epsilon} \int_{\bdy \Omega} u^- \rev{\pd_\nu u^-} d\sigma.
\end{equation}
Apply Green's first identity for $\Omega^+$ and $\Omega^-_R$,
use the Helmholtz equation for $U^\pm$ and the radiation condition for $U^-$, and multiply the equation by ${\hat\epsilon}k_-/i$ to obtain 
\begin{multline}   \label{eq:greeninjhelm}
\frac{{\hat\epsilon} k_-}{k_+}\left(\frac 1i\int_{\Omega^+}(k_+|\nabla U^+|^2-\rev k_+|k_+U^+|^2) dx\right) 
\\+ |{\hat\epsilon}|^2\left( \frac 1i\int_{\Omega^-_R} (k_-|\nabla U^-|^2-\rev k_-|k_-U^-|^2) dx
+\int_{|x|=R} |k_-U^-|^2 d\sigma\right)\to 0,
\end{multline}
as $R\to\infty$.
Denote the three integrals appearing in \eqref{eq:greeninjhelm} by $I^+$, $I^-_R$ and $I_R$ respectively, including $i^{-1}$ in the first two,
and set $\phi_\pm:=|\arg(k_\pm/i)|$ as in Definition~\ref{defn:hexagon},
let $z:= {\hat\epsilon} k_-/k_+$,
and define the sector
\begin{equation}
   S_\phi:= \sett{z\in\C\setminus\{0\}}{|\arg(z)|\le \phi} \cup \{0\}. 
\end{equation}
We note that 
\begin{equation}   \label{eq:Iplusexpr}
  I^+= \im k_+\int_{\Omega^+} (|\nabla U^+|^2+|k_+U^+|^2) dx
  -i\re k_+\int_{\Omega^+} (|\nabla U^+|^2-|k_+U^+|^2) dx,
\end{equation}
and similarly for $I^-_R$.

\begin{figure}[t]
  \centering
  \includegraphics[height=50mm]{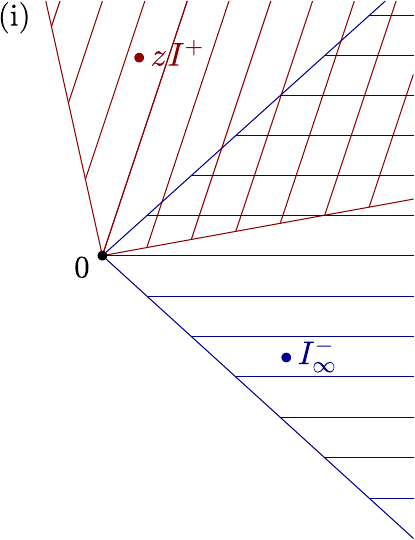}
  \includegraphics[height=50mm]{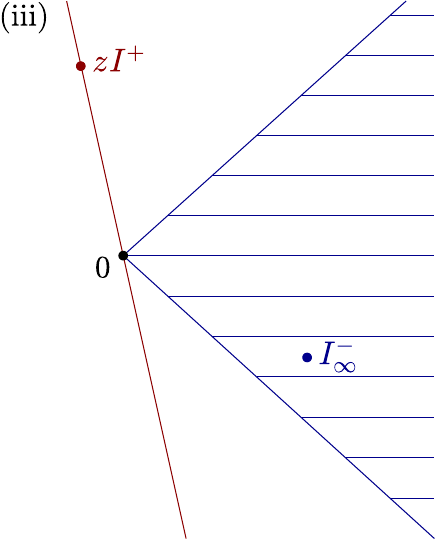}
  \includegraphics[height=50mm]{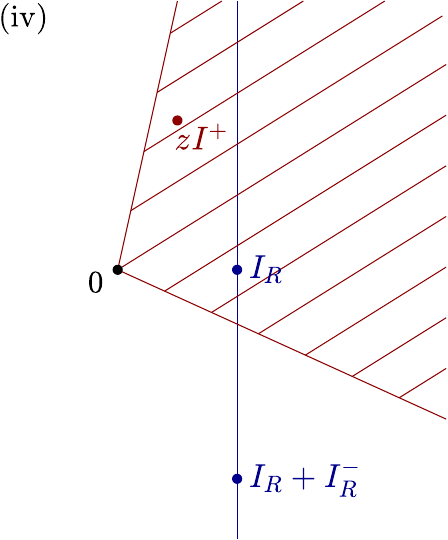}
  \caption{\sf Sectors and lines appearing, depending on $\phi_\pm$.}
\label{fig:sectors}
\end{figure}

We verify that the condition \eqref{eq:injcondforeps} implies $U^\pm=0$, by examining \eqref{eq:greeninjhelm} in the nine
cases $\phi_\pm=0$, $0<\phi_\pm<\pi/2$ and $\phi_\pm=\pi/2$ as follows.

(i) Assume $0<\phi_-<\pi/2$ and $0<\phi_+<\pi/2$. Then $U^-$ decays
exponentially as $R\to\infty$. Setting $R=\infty$ in
\eqref{eq:greeninjhelm}, we have the equation $z
I^++|{\hat\epsilon}|^2I^-_\infty=0$, where $I^+\in S_{\phi_+}$ and
$I^-_\infty\in S_{\phi_-}$. If $|\arg(z)|+\phi_-+\phi_+<\pi$, then this
is possible only if $I^+=I^-_\infty=0$. See Figure~\ref{fig:sectors}(i).
Indeed, the sectors
$zS_{\phi_+}$ and $-S_{\phi_-}$ intersect only at $0$. We conclude in
particular that $\re I^+=0=\re I^-_\infty$, and therefore $U^+=U^-=0$
according to \eqref{eq:Iplusexpr}.

If $|\arg(z)|+\phi_-+\phi_+=\pi$, then the rotated sectors
$zS_{\phi_+}$ and $-S_{\phi_-}$ touch and we observe that $I^+$ and
$I^-_\infty$ lie on the boundary of respective sector $S_{\phi _\pm}$.
From \eqref{eq:Iplusexpr}, we conclude that  $\int_{\Omega^+}(|\nabla
U^+|^2-|k_+U^+|^2)dx=\pm \int_{\Omega^+}(|\nabla U^+|^2+|k_+U^+|^2)dx$.  This forces either
$U^+=0$ or $\nabla U^+=0$. We conclude from the Helmholtz equation
that $U^+=0$. A similar argument shows that $U^-=0$.

(ii) Assume $\phi_\pm<\pi/2$ and either $\phi_+=0$ or $\phi_-=0$.
When $|\arg(z)|+\phi_-+\phi_+<\pi$, the argument in (i) applies. When
$|\arg(z)|+\phi_-+\phi_+=\pi$ and $\phi_+\ne 0$, the argument in (i)
shows that $U^+=0$. The jump condition then shows that $U^-=0$. If
instead $|\arg(z)|+\phi_-+\phi_+=\pi$ and $\phi_-\ne 0$, a similar
argument shows that $U^-=0$, from which $U^+=0$ follows from the jump
conditions.

(iii) Assume $\phi_-<\pi/2$ and $\phi_+=\pi/2$. We now have the
equation $z I^++|{\hat\epsilon}|^2I^-_\infty=0$, with $\re I^+=0$ and
$I^-_\infty\in S_{\phi_-}$. If $\min(|\arg(z)|,|\arg(-z)|)<
\pi/2-\phi_-$, then the line $zS_{\phi_+}$ intersects the sector
$-S_{\phi_-}$ only at the origin, which shows that $I^+=I^-_\infty=0$.
See Figure~\ref{fig:sectors}(iii).
From the exterior ($\Omega^-$) analogue of \eqref{eq:Iplusexpr}, we
conclude $U^-=0$, and jump conditions imply $U^+=0$.

If $\min(|\arg(z)|,|\arg(-z)|)= \pi/2-\phi_-$ and if $\phi_->0$, then
it follows that $I^-_\infty$ lies on the boundary of $S_{\phi_-}$. As
in (i), this shows that $U^-=0$, and by jump conditions that $U^+=0$.

(iv) Assume $\phi_-=\pi/2$. Then $\re I^-_R=0$, and
\eqref{eq:greeninjhelm} reduces to $\re(z I^+)+ \lim_{R\to
  \infty}I_R=0$. If $|\arg(z)|+\phi_+\le \pi/2$, then $\re(z I^+)\ge
0$ and $\lim_{R\to \infty} I_R=0$ follows. 
See Figure~\ref{fig:sectors}(iv).
By Rellich's lemma this
implies that $U^-=0$. The jump relations then shows that $U^+=0$.

If also $\phi_+=\pi/2$ then $\re I^+=0$, and we conclude that
$\lim_{R\to \infty} I_R=0$ also when $z<0$, and can in the same way
conclude that $U^-=0=U^+$.
\end{proof}

Next consider the existence of solutions $U^\pm$. 
The following result is essentially from \cite{AxThesisPub4:06}, where
more details and background can be found. 
 For a short survey of the Fredholm theory that we apply, 
we refer to \cite[Sec. 6.4]{RosenGMA:19}. 

\begin{prop}   \label{prop:helmexist}
    Let $\Omega^+\subset \R^n$ be a bounded Lipschitz domain
    with connected exterior $\Omega^-$. 
  Then there exists $1\le C(\bdy \Omega)<\infty$ such that
   if
\begin{equation}  \label{eq:helmexistcond}
  {\hat\epsilon}\in \C\setminus [-C(\bdy \Omega), -1/C(\bdy \Omega)] \quad\text{and}\quad {\hat\epsilon} /\hat k\in \swp(k_-,k_+),
\end{equation} 
recalling that $\hat k= k_+/k_-$ and Definition~\ref{defn:hexagon}, 
then there exists a unique solution $U^+\in H^1(\Omega^+)$, $U^-\in H^1(\Omega^-_R)$ 
to the Helmholtz transmission problem
   \eqref{eq:Helmtransmprobl} in $\R^n$, $n\ge 2$,  depending continuously
 on the  datum $u^0\in H^{1/2}(\bdy \Omega)$.
\end{prop}

\begin{proof}
   (i) We first use Fredholm theory to reduce the problem to an estimate of
 \begin{equation}  \label{eq:helmneededest}
  \int_{\Omega^+}|\nabla U^+|^2 dx+\int_{\Omega^-_R}|\nabla U^-|^2  dx
\end{equation}
by the $\mH$ norm of $u^0$ and compact terms.
To this end, it is convenient to consider the multivector fields
  $F^\pm=\nabla U^\pm+ik_\pm U^\pm$ and $F^0= \nabla U^0+ik_-U^0$
  as in \eqref{eq:helmgrad}.
  For fixed $k_\pm$, we define function spaces $\widetilde \mH_{k_\pm}^{1,\pm}$
  of such $F^\pm$  in $\Omega^\pm$, with potentials $U^\pm$ solving $\Delta U^\pm+ k_\pm^2 U^\pm=0$, and $U^-$ satisfying the radiation condition in $\Omega^-$.
  The norm of $F^+\in \widetilde \mH^{1,+}_{k_+}$ is $\|F^+\|_{L_2(\Omega^+)}$
  and the norm of $F^-\in \widetilde \mH^{1,-}_{k_-}$ is $\|F^-\|_{L_2(\Omega^-_R)}$,
  with a fixed large $R$. 
 For the data/right-hand sides, we define the subspace 
$\mH^1_{k_-}\subset \mH$ consisting of $f= f_0+f_{1\ta}+ f_{1\no}$,
with $f_0$ scalar and  $f_{1\ta}, f_{1\no}$ tangential and normal vector fields 
respectively, satisfying the constraint $\nabla_\ta f_0= ik_-f_{1\ta}$.
The traces of incoming waves $F^0$ all belong to $\mH^1_{k_-}$. 

The transmission problem \eqref{eq:Helmtransmprobl}  amounts 
to inverting  a bounded 
linear map
\begin{equation}
   T_{\hat\epsilon, k_+, k_-}: \widetilde \mH_{k_+}^{1,+}
   \oplus \widetilde \mH_{k_-}^{1,-}\to \mH^1_{k_-}.
\end{equation}
Assume first that $k_+=k_-$ and that $\hat\epsilon=1$.
In this case the jump condition in \eqref{eq:Helmtransmprobl} reduces to
$f^+-f^-=f^0$, and it follows that $T_{\hat\epsilon, k_+, k_-}$ is invertible
since the Cauchy integrals \eqref{eq:fieldCauchy}, with $x\in\Omega^\pm$, provide 
an explicit inverse.

Next consider general parameters $\hat\epsilon, k_+, k_-$ satisfying
\eqref{eq:helmexistcond}.
It suffices to show that $T_{\hat\epsilon, k_+, k_-}$
is a Fredholm operator with index zero. 
Indeed, Proposition~\ref{prop:helminj} shows injectivity, from which surjectivity then follows.
To prove index zero, we note that 
$\C\setminus [-C(\bdy \Omega), -1/C(\bdy \Omega)]$
is an open connected set, and aim to apply Fredholm perturbation theory,
perturbing $\hat\epsilon$ to $1$ and $k_+$ to $k_-$.
We prove in (ii) below that the operator $T_{\hat\epsilon, k_+, k_-}$
is semi-Fredholm whenever $\hat\epsilon\notin [-C(\bdy \Omega), -1/C(\bdy \Omega)]$, so it remains to verify that the operator and function spaces depend
continuously on the parameters. 

From Hodge decompositions of the space $\mH$, see 
\cite{AxThesisPub4:06}, it follows that for
$k_-\ne 0$ there are projections $\mH\to \mH^1_{k_-}$ onto these subspaces,
depending continuously on $k_-$.
This uses the Hodge projections not for the exterior and 
interior derivatives, but for zero order perturbations of these defined by the
wave number $k_-$ (c.f. \eqref{eq:hodgekpotential} below).
The crucial observation is that for all $k_-\ne 0$, the cohomology for
this perturbed Hodge decomposition vanishes. This implies that the
perturbed Hodge projections depend continuously on $k_-$.

For the domain spaces $\widetilde \mH_{k_\pm}^{1,\pm}$, we note that 
the trace map and the Cauchy integral give an isomorphism between
$F^\pm\in\widetilde \mH_{k_\pm}^{1,\pm}$ and 
$f^\pm\in\mH_{k_\pm}^{1,\pm}\subset E_k^\pm\mH\subset \mH$.
Like for $\mH^1_{k_-}$, it follows from Hodge decompositions of $\mH$
that for
$k_\pm\ne 0$ there are projections $\mH\to \mH_{k_\pm}^{1,\pm}$ onto these 
trace spaces for Helmholtz fields,
depending continuously on $k_\pm$.
Since $T_{\hat\epsilon, k_+, k_-}$ clearly depend continuously on $\hat\epsilon$,
perturbation theory applies to show index zero, provided we show the
estimate \eqref{eq:helmneededest}.

(ii) 
To establish the estimate  of  \eqref{eq:helmneededest}, we construct certain auxiliary potentials
$V^\pm$
to the gradient vector fields $\nabla U^\pm$.
In $\Omega^-_R$, we simply use the given scalar potential $V^-=U^-$.
In $\Omega^+$, we find a bivector field $V^+:\Omega^+\to\wedge^2\R^n$ and a vector field $\tilde V^+:\Omega^+\to \wedge^1\R^n$
such that 
\begin{equation}  \label{eq:hodgekpotential}
\begin{cases}
  \nabla U^+= \nabla\lctr V^++ ik_+ \tilde V^+,\\
  ik_+ U^+= \nabla\lctr \tilde V^+.
\end{cases}
\end{equation}
(Modulo the term $\tilde V^+$, this means that $V^+$ is a conjugate function
when $n=2$ and a vector potential when $n=3$.)
The existence and compactness of the 
map $H^1(\Omega^+)\to L_2(\Omega^+)^2: U^+\mapsto (V^+,\tilde V^+)$
follows from Hodge decompositions as in \cite{AxMcIntosh:04}, 
after translation
from the spacetime framework using \cite[Sec. 9.1]{RosenGMA:19}.
To complete the construction of $V^\pm$, we extend the potentials
$V^\pm$
to compactly supported functions on $\R^n$, with $\nabla V^-$ and 
$\nabla\lctr V^+$ belonging to $L_2(\R^n)$.

Pairing the jump relations with $v^\pm$, we have
\begin{equation}  \label{eq:helmjumpints}
\begin{cases}
\int_{\bdy \Omega}\scl{\nu\wedg\nabla u^+}{v^+} d\sigma=
\int_{\bdy \Omega}(\scl{\nu\wedg\nabla u^-}{v^+}+\scl{\nu\wedg\nabla u^0}{v^+}) d\sigma,\\
\int_{\bdy \Omega}\scl{v^-}{\nu\lctr\nabla u^+} d\sigma
=\rev{\hat\epsilon} \int_{\bdy \Omega}(\scl{v^-}{\nu\lctr\nabla u^-}+
\scl{v^-}{\nu\lctr\nabla u^0})d\sigma.
\end{cases}
\end{equation}
Using the general Stokes theorem, see \cite[Sec. 7.3]{RosenGMA:19}, 
we have modulo compact terms
\begin{align}
  \int_{\bdy \Omega}\scl{\nu\wedg\nabla u^+}{v^+}d\sigma &\approx
  \int_{\Omega^+}|\nabla U^+|^2 dx,\\ 
  \int_{\bdy \Omega}\scl{v^-}{\nu\lctr\nabla u^-}d\sigma &\approx
  -\int_{\Omega^-_R}|\nabla U^-|^2 dx,
\end{align}
and
\begin{multline}
  \int_{\bdy \Omega}\scl{v^-}{\nu\lctr\nabla u^+}d \sigma-
  \int_{\bdy \Omega}\scl{\nu\wedg\nabla u^-}{v^+} d\sigma\\
  \approx  
  \int_{\Omega^+}\scl{\nabla V^-}{\nabla U^+} dx+
  \int_{\Omega^-_R}\scl{\nabla U^-}{\nabla\lctr V^+}dx
  \approx\int_{\R^n}\scl{\nabla V^-}{\nabla\lctr V^+}dx=0.
\end{multline}
Therefore, adding the equations \eqref{eq:helmjumpints} yields
an estimate of \eqref{eq:helmneededest} whenever we do not have
${\hat\epsilon}<0$.
When ${\hat\epsilon}$  is negative and close to $\infty$,  we instead subtract the equations to conclude.
When ${\hat\epsilon}$  is negative and close to $0$,  we can also obtain an estimate of 
\eqref{eq:helmneededest} by instead starting with a bivector potential
$V^+$ in $\Omega^+$, and a scalar potential $V^-$ in $\Omega^-$.
\end{proof}

\section{The 2D Dirac integral}  \label{sec:diracint2d}

In this section, we derive the Dirac integral equation \eqref{eq:2DDiracequ}
in 2D by combining two Helmholtz problems, and using a duality ansatz.
We start by recalling how 2D Maxwell transmission problems reduce to
Helmholtz transmission problems. 
We end by optimizing the Dirac parameters $r,\beta,\alpha',\beta'$, which is a main
step in the construction of the Dirac BIE \eqref{eq:2DDiracequ}.

\begin{ex}[Transverse magnetic (TM) scattering]  \label{ex:TMMaxw}
  We consider applications of the Dirac
  equation to the scattering of EM fields as in \eqref{eq:totalEM},
  but independent of the $e_3$-coordinate and polarized so that
\begin{equation}   \label{eq:totalEM:TM}
  F_\emm(x)= E_1(x)e_1+E_2(x)e_2+
  B_3(x) e_{12}, \qquad x\in \R^2.
\end{equation}
To write a Helmholtz equation for this EM field, we normalize by a
left Clifford multiplication and define the field
\begin{equation}
  F= F_\emm(-i\epsilon e_{12}),
\end{equation}
where $\pm$ is suppressed. Since the Clifford product is associative,
we have
\begin{equation}
  (\dirac-ik)F= ((\dirac-ik)F_\emm)(-i\epsilon e_{12})=0.
\end{equation}
Writing $F=ikU+V_1e_1+V_2e_2$, the Dirac equation $\dirac F=ikF$
amounts to $V=\nabla U$ and $\divv V= -k^2 U$, that is the Helmholtz
equation $\Delta U +k^2 U=0$ for the scalar function $U$. We have
\begin{align}
      F&= ikU+\nabla U,\\
      E&= (i\epsilon)^{-1} (\nabla U) e_{12},\\
      B_3&= (k/\epsilon) U.
\end{align}
With this setup for both domains $\Omega^\pm$, 
jump relations for the electromagnetic field specify
the jump matrix
\begin{equation}
  M= \diag\begin{bmatrix} \hat k & a & \hat\epsilon & 1 \end{bmatrix}
\end{equation}
for $F^\pm$ in the frame $\{1,\nu\tau, \nu,\tau\}$. The parameter $a$
can be chosen freely since $F_2=0$ for the field $F$.
\end{ex}

The following Dirac well-posedness result exploits that the Dirac
equation in the plane consists of two coupled Helmholtz equations.

\begin{prop}   \label{prop:diractrans}
  Consider the Dirac transmission problem \eqref{eq:diractranspr}  for 
  a bounded Lipschitz domain $\Omega^+\subset \R^2$ with connected 
  exterior $\Omega^-$. 
  Let
\begin{equation}
  M= \diag\begin{bmatrix} \hat k & \hat k/\beta & \alpha & 1 \end{bmatrix},
\end{equation}
with parameters $\beta,\alpha\in\C\setminus\{0\}$. Assume that
\begin{equation}
  \alpha,\beta\in \C\setminus [-C(\bdy \Omega), -1/C(\bdy \Omega)] \quad\text{and}\quad  \alpha/\hat k, \beta/\hat k\in \swp(k_-,k_+),
\end{equation} 
recalling that $\hat k= k_+/k_-$ and Definition~\ref{defn:hexagon}. 
Then the operator $B_2: E_{k_+}^+\mH_2\oplus E_{k_-}^-\mH_2 \to \mH_2$ given by
\begin{equation}
 B_2(f^+,f^-):= f^+-Mf^-
\end{equation}
is invertible,  where we as in Section~\ref{sec:tracespace} 
denote the spaces of traces
of solutions on $\bdy \Omega$ by $E_{k_\pm}^\pm\mH_2$. 
\end{prop}

\begin{proof}
Consider the equation
\begin{equation}
   f^+ =Mf^-+ Mf^0
\end{equation}
with a given source $f^0$, and write $f^+= f_0^++f^+_1+ f^+_2 e_{12}$
with scalar functions $f_0^+$ and $f^+_2$ and a vector field $f^+_1$,
and similarly for $f^-$ and $f^0$.
We first prove uniqueness in two steps as follows. To this end we assume 
that $f^0=0$.

(i) The scalar functions $F^\pm_2$ solve the Helmholtz equation as a
consequence of $F^\pm$ solving the Dirac equation, with wavenumbers
$k_\pm$ respectively. Moreover, the vector part of $\dirac F^\pm=
ik_\pm F^\pm$ shows that
\begin{equation}  \label{eq:diracmid}
  \nabla F^\pm_2= (\nabla F^\pm_0-ik_\pm F^\pm_1)e_{12}.
\end{equation}
From the assumed jump relations for $f^\pm$ it therefore follows that
$f_2^+= \hat kf^-_2/\beta$ and $\pd_\nu f_2^+= \beta\pd_\nu(\hat
kf^-_2/\beta)$, and hence Proposition~\ref{prop:helminj} with
${\hat\epsilon}=\beta$ shows that $f_2^\pm=0$.

(ii) Next consider the scalar functions $F^\pm_0$, which also solve
the Helmholtz equation. Since $F^\pm_2=0$ by (i), we have $\nabla
F^\pm_0=ik_\pm F^\pm_1$ and obtain jump relations $f_0^+= \hat kf^-_0$
and $\pd_\nu f_0^+= \alpha\pd_\nu(\hat kf^-_0)$. Again
Proposition~\ref{prop:helminj} applies, now with
${\hat\epsilon}=\alpha$, and shows that $f_0^\pm=0$. From
\eqref{eq:diracmid} we conclude $f^\pm_1=0$ and in total $f^\pm=0$.

To show existence, it suffices by perturbation theory for Fredholm operators
to prove an estimate
\begin{equation} 
  \|f^+\|_{\mH_2}+ \|f^-\|_{\mH_2}\lesssim \|f^0\|_{\mH_2}.
\end{equation} 
This follows as in steps (i) and (ii) but instead appealing to
Proposition~\ref{prop:helmexist}. In this case, we obtain in step (i)
that $\|f^+_2\|_{\mH_2}+\|f^-_2\|_{\mH_2}\lesssim \|f^0\|_{\mH_2}$,
which in step (ii) is used to estimate $f^\pm_0$ and $f^\pm_1$.
\end{proof}

The next result is central to this paper, where we derive Dirac integral equations
by using an ansatz obtained from an auxiliary Dirac transmission problem via duality.

\begin{prop}   \label{prop:diracintwp}
  Assume  the hypothesis of Proposition~\ref{prop:diractrans}, 
and further assume that
\begin{equation}
  M'= \diag\begin{bmatrix} \alpha' & 1 & 1/\hat k & 1/(\beta' \hat k) \end{bmatrix},
\end{equation}
with parameters satisfying 
\begin{equation}
  \alpha',\beta'\in \C\setminus [-C(\bdy \Omega), -1/C(\bdy \Omega)] \quad\text{and}\quad \alpha'\hat k,\beta' \hat k\in \swp(k_+,k_-).
\end{equation}
Then the operator 
\begin{equation}   \label{eq:predirint}
  rE_{k_+}^+M'+M E_{k_-}^-
\end{equation}
is invertible on $\mH_2$, for any $r\in\C\setminus\{0\}$.
\end{prop}

\begin{proof}
  We factorize \eqref{eq:predirint} via $E_{k_+}^+\mH_2\oplus  E_{k_-}^-\mH_2$ as
\begin{equation}   \label{eq:factors2d}
  \begin{bmatrix} E_{k_+}^+ & -M E_{k_-}^-\end{bmatrix}
  \begin{bmatrix} r & 0 \\ 0 & 1 \end{bmatrix}
  \begin{bmatrix} E_{k_+}^+M' \\- E_{k_-}^- \end{bmatrix}.
\end{equation}  
By Proposition~\ref{prop:diractrans}, the left factor is invertible, so
it suffices to show that the right factor also is invertible. 
To this end, we use the (non-Hermitean) complex bi-linear duality 
$$
(f,g)_\mH= \int_{\bdy \Omega} \scl{\nu(x)\inv{f}(x)}{g(x)} d\sigma(x)
$$
on $\mH_2$, where $\inv w= \sum_j (-1)^j w_j$ denotes the involution
of a multivector $w=\sum_j w_j$, $w_j\in\wedge^j \R^n$. 
It is readily verified that
\begin{align}
  (E_{k_\pm} f, g)_\mH &=  -(f,E_{k_\pm}  g)_\mH,\\
   (M' f, g)_\mH &=  (f,\widetilde M g)_\mH,
\end{align}
where $\widetilde M=\diag\begin{bmatrix}  1/\hat k & (1/\hat k)/\beta' & \alpha' & 1 \end{bmatrix}$.
This shows that in the natural way $E^\mp_{k_\pm}\mH_2$ is the 
dual space of $E^\pm_{k_\pm}\mH_2$.
Hilbert space duality theory shows that in \eqref{eq:factors2d}, invertibility of
the right factor is equivalent to invertibility of the left factor,
with $M$ replaced by $\widetilde M$, and $k_-$ and $k_+$ swapped.
\end{proof}

Consider the Dirac integral equation 
\begin{equation}  \label{eq:gendirint}
  (rE_{k_+}^+M'+M E_{k_-}^-)\tilde h= Mf^0,
\end{equation}
involving the operator from \eqref{eq:predirint},
with right-hand side specified by the jump condition in \eqref{eq:diractranspr}
and auxiliary density $\tilde h$ which we precondition as $\tilde h= P'h$ in
the proof of Theorem~\ref{thm:2Ddiracint}. 
We optimize \eqref{eq:gendirint}  by 
choosing the parameters $r,\beta,\alpha',\beta'$.
Recall that for EM fields $\alpha= \hat\epsilon$,
where our main interest is $\im\hat\epsilon\ge 0$ and $\hat k= \sqrt{\hat \epsilon}$, so that $\re\hat k\ge 0$.
We therefore consider $\alpha$ as having a prescribed value.
Clearly 
\begin{equation}  \label{eq:reflopexp}
  rE_{k_+}^+M'+M E_{k_-}^-= \tfrac 12(rM'+M+ rE_{k_+} M'-M E_{k_-}),
\end{equation}
where 
\begin{equation}  
 rM'+M= 
  \diag\begin{bmatrix} r\alpha'+\hat k &  r+\hat k/\beta & r/\hat k+\alpha & r/(\beta' \hat k)+1 \end{bmatrix}.
\end{equation}
 Let $K^v$ and $S^a$ be the static double and single layer type operators,
that is $K^v= K^v_0$, and $S^a$ is $S^a_k$ without the factor $ik$ at $k=0$.
Modulo operators of the form  $K^v_{k_\pm}-K^v$,
$S^a_{k_+}-ik_-\hat kS^a$ and $S^a_{k_-}-ik_-S^a$,  the integral operator $T=rE_{k_+} M'-M E_{k_-}$ from \eqref{eq:reflopexp}
is the entry-wise product of  
\begin{equation}    \label{eq:prediracintop2d}
  \begin{bmatrix}
   r\alpha'-\hat k & r-\hat k & r-\hat k & r/\beta'-\hat k \\
   r\alpha' -\hat k/\beta & r-\hat k/\beta & r-\hat k/\beta & r/\beta'-\hat k/\beta  \\
   \hat kr\alpha'-\alpha & \hat k r-\alpha & r/\hat k-\alpha  & r/(\beta'\hat k)-\alpha \\
   \hat k r\alpha'- 1 &  \hat k r-1 &  r/\hat k -1 & r/(\beta'\hat k)-1
  \end{bmatrix}.
\end{equation} 
and $E_k$, 
with diagonal and off-diagonal $2\times 2$ blocks replaced by $K^v$ and
$ik_- S^a$ respectively. 
 Indeed, under these approximations
\begin{multline}  
  T\approx
  \begin{bmatrix}
    -K^{\nu'} & -K^{\tau'} & \hat k(ik_-S^1) & 0 \\
    K^{\tau'} & -K^{\nu'} & 0 & \hat k(ik_-S^1) \\
    \hat k(ik_-S^{\nu\cdot\nu'}) & \hat k(ik_-S^{\nu\cdot\tau'}) & -K^\nu & K^\tau \\
     \hat k(ik_-S^{\tau\cdot\nu'}) &  \hat k(ik_-S^{\tau\cdot\tau'}) & - K^\tau & -K^\nu
  \end{bmatrix}(rM') \\
-M
  \begin{bmatrix}
    -K^{\nu'} & -K^{\tau'} & ik_-S^1 & 0 \\
    K^{\tau'} & -K^{\nu'} & 0 & ik_-S^1 \\
    ik_-S^{\nu\cdot\nu'} & ik_-S^{\nu\cdot\tau'} & -K^\nu & K^\tau \\
     ik_-S^{\tau\cdot\nu'} &  ik_-S^{\tau\cdot\tau'} & - K^\tau & -K^\nu
  \end{bmatrix}.
\end{multline}

\begin{itemize}
\item Our first choice is to set $r=\hat k$. 
This gives cancellation in the (1,2) and (4,3) elements of the operator 
$T$, which for a smooth domain $\Omega$ yields
 $T^2=0$  on $\mH_2$ modulo compact operators.
In particular the essential spectrum of $T$ is $\{0\}$.

\item Our choices for $\beta, \beta', \alpha'$ are to set
$\beta= \hat k/|\hat k|$ and $\beta'=\alpha'= \rev {\hat k}/|\hat k|$.
These choices make $\beta/\hat k>0$, $\beta' \hat k>0$ and $\alpha' \hat k>0$.
Therefore, if $\alpha\in \C\setminus  [-C(\bdy \Omega), -1/C(\bdy \Omega)]$ and $\alpha/\hat k\in\swp(k_-,k_+)$,
then Propositions~\ref{prop:diractrans} and \ref{prop:diracintwp}
guarantee invertibility of $E_{k_+}^+M'+M E_{k_-}^-$.

Furthermore, the choice of $\beta, \beta',\alpha'$
yields diagonal (1,1), (2,2) and (4,4) elements in \eqref{eq:predirint}
which are compact perturbations of invertible operators  on any Lipschitz domain, whenever $\hat k\in\C\setminus(-\infty,0]$.
Indeed,  $K^\nu_{k_\pm}-K^\nu_0$  is compact and the essential spectrum
of $K^\nu_0$ is contained in $(-1,1)$.
Moreover, when  $\re \hat k\ge 0$  then the
normalization $|\beta|= |\beta'|=|\alpha'|=1$
gives spectral points $\lambda$ for $K^\nu_0$ on the imaginary axis
with $|\lambda|\ge 1$, since
$\lambda= (1+z)/(1-z)$
maps $\{|z|=1,\re z\ge 0\}$ onto $\{\re \lambda=0, |\lambda|\ge 1\}$.
\end{itemize}

To summarize, for solving the Helmholtz/TM Maxwell transmission problem 
as described above, we have obtained the 2D Dirac integral equation
\begin{equation}  \label{eq:2DfinalDirac}
  (\hat kM'+M+  E_{k_+} (\hat kM')-M E_{k_-})\tilde h= 2M f^0,
\end{equation}
with 
\begin{align}
  M&= \diag\begin{bmatrix} \hat k & |\hat k| & \hat \epsilon & 1 \end{bmatrix},\\
  \hat k M'&= \diag\begin{bmatrix} |\hat k| & \hat k & 1 & \hat k/|\hat k| 
  \end{bmatrix}.
\end{align}

\section{The 3D Dirac integral}   \label{sec:3DDirac}

In Sections~\ref{sec:helmwp} and \ref{sec:diracint2d}, we derived an
integral equation for solving Dirac transmission problems in $\R^2$,
which applies to Helmholtz/TM Maxwell scattering. We here derive the
completely analogous integral equation in $\R^3$, with applications to
scattering for the full Maxwell equations and not only the Helmholtz
equation. 
This Dirac integral equation in 3D combines one Maxwell problem and 
two Helmholtz problems, and uses a duality ansatz. 
We end this section by optimizing the Dirac parameters $r,\beta,\gamma,\alpha',\beta',\gamma'$, 
which is a main step in the construction of the Dirac BIE \eqref{eq:3DDiracequ}.

\begin{ex}
  Maxwell's equations correspond to an electromagnetic field $F$ with
  $F_0=0=F_3$ as in \eqref{eq:totalEM}. The energy norm that we
  consider is simply the $L_2$ norm of $F^\pm$ in $\Omega^+$ and
  $\Omega^-_R$, respectively, and the corresponding function space
  $\mH_3$ on $\bdy \Omega$ is $(H^{-1/2}(\bdy\Omega))^6$, with
  tangential curls of $E$ and $B$ belonging to $H^{-1/2}(\bdy\Omega)$.

In the 3D Dirac transmission problem \eqref{eq:diractranspr}, Maxwell
scattering for the field $F=F_\emm$ from \eqref{eq:totalEM} specifies
the jump matrix
\begin{equation}  
 M= 
  \diag\begin{bmatrix} a  & 1/\hat k  & \hat k/\hat\epsilon
  & \hat k/\hat\epsilon
  & b & \hat\epsilon^{-1} & 1 & 1  \end{bmatrix}
\end{equation}
in the frame \eqref{eq:3dframe},
by the jump relations for the electric and magnetic field.
The parameters $a$ and $b$ can be chosen freely since $F_0=0= F_3$.
\end{ex}

Consider the Maxwell transmission problem \eqref{eq:maxwtransprclassi}, which 
we write in multivector notation, with $F_1^\pm= E^\pm$ and $F_2^\pm$
being the Hodge dual of $B^\pm$, as follows.
\begin{equation}  \label{eq:maxwtranspr}
\begin{cases}
   \nu\wedg f^+_1= \nu\wedg (f^-_1+f^0_1), & x\in \bdy\Omega,\\
   \nu\lctr f^+_2=(\hat k/\hat \epsilon) \nu\lctr (f^-_2+f^0_2), & x\in \bdy\Omega,\\
   \dirac F^+= ik_+ F^+, F^+_0=F^+_3=0, &x\in\Omega^+,\\
   \dirac F^-= ik_- F^-, F^-_0=F^-_3=0, &x\in\Omega^-,\\
   (x/|x|-1)F^-= o(|x|^{-1}e^{\im k_- |x|}), & x\to\infty.
   \end{cases}
\end{equation}

We require the following Maxwell versions of the results in 
Section~\ref{sec:helmwp}.

\begin{prop}   \label{prop:maxwinj}  
  Consider the Maxwell transmission problem \eqref{eq:maxwtranspr} for 
  a bounded Lipschitz domain $\Omega^+\subset \R^3$ with connected 
  exterior $\Omega^-$. 
Assume that the incoming wave vanishes
so that $f^0=0$.
If  $\hat\epsilon$ from  \eqref{eq:maxwtranspr}
satisfies 
\begin{equation}
{\hat\epsilon} /\hat k \in \swp(k_-,k_+),
\end{equation}

recalling that $\hat k= k_+/k_-$ and Definition~\ref{defn:hexagon}, 
then $F^+=F^-=0$ identically.
\end{prop}

\begin{proof}
  Similar to the proof of Proposition~\ref{prop:helminj}, we use
  the jump relations to obtain
\begin{equation}
  \int_{\bdy \Omega}\scl{f_1^+}{\nu\lctr f_2^+}d\sigma = \rev {\hat k/\hat \epsilon} \int_{\bdy \Omega} \scl{f_1^-}{\nu\lctr f_2^-} d\sigma.
\end{equation}
We then apply a Stokes theorem for $\Omega^+$ and $\Omega^-_R$,
to obtain  
\begin{multline}     \label{eq:maxwident}
\frac{\hat k}{\hat \epsilon}\left(\frac 1i\int_{\Omega^+} (k_+|F_2^+|^2-\rev k_+|F_1^+|^2) dx\right)
\\+ \frac{|{\hat k}|^2}{|\hat \epsilon|^2}\left( \frac 1i\int_{\Omega^-_R} (k_-|F_2^-|^2-\rev k_-|F_1^-|^2) dx
+\tfrac 12\int_{|x|=R} |F^-|^2 d\sigma \right)\to 0,
\end{multline}
as $R\to\infty$. We here used that $\nabla\lctr F^\pm_2= ik_\pm
F^\pm_1$, $\nabla\wedg F^\pm_1= ik_\pm F^\pm_2$, and by the radiation
condition that $\scl{F_1^-}{\nu\lctr F_2^-}\approx |F^-_1|^2\approx
|F^-_2|^2\approx\tfrac 12|F^-|^2$ on $|x|=R$. Using
\eqref{eq:maxwident}, the result follows similarly to the proof of
Proposition~\ref{prop:helminj}.
\end{proof}

\begin{prop}   \label{prop:maxwexist}
    Let $\Omega^+\subset \R^3$ be a bounded Lipschitz domain
    with connected exterior $\Omega^-$. 
     Then there exists
  $1\le C(\bdy\Omega)<\infty$ such that if
   $\hat\epsilon$ from  \eqref{eq:maxwtranspr}
satisfies 
\begin{equation}   \label{eq:maxwcondforexist}
  {\hat\epsilon},  \hat k^2/\hat \epsilon  \in \C\setminus [-C(\bdy \Omega), -1/C(\bdy \Omega)] \quad\text{and}\quad {\hat\epsilon}/\hat k \in \swp(k_-,k_+),
\end{equation}
recalling that $\hat k= k_+/k_-$ and Definition~\ref{defn:hexagon}, 
then there exists a unique solution $F^+\in L_2(\Omega^+)$, $F^-\in
L_2(\Omega^-_R)$ to the Maxwell transmission problem
\eqref{eq:maxwtranspr},  depending contiuously on the 
trace $f^0\in \mH_3$ of the incoming 
electromagnetic wave $F^0$. 
\end{prop}

\begin{proof}
 (i)
The proof is similar to that of Proposition~\ref{prop:helmexist},
replacing the scalar function $ik_\pm U^\pm$ by the vector field $F^\pm_1$,
and  the vector field $\nabla U^\pm$ by  the bivector field $F^\pm_2$.
We first use Fredholm theory to reduce the problem to an estimate of
\begin{equation}  \label{eq:maxwneededest}
  \int_{\Omega^+}|F^+|^2 dx+\int_{\Omega^-_R}|F^-|^2 dx 
\end{equation}
by the norm of $f^0$ and compact terms.
We define function spaces $\widetilde \mH^{2,\pm}_{k_\pm}$ of 
Maxwell fields $F^\pm= F^\pm_1+F^\pm_2$ in $\Omega^\pm$
($F^-$ satisfying the radiation condition in $\Omega^-$).
  The norm of $F^+\in \widetilde \mH^{2,+}_{k_+}$ is $\|F^+\|_{L_2(\Omega^+)}$
  and the norm of $F^-\in \widetilde \mH^{2,-}_{k_-}$ is $\|F^-\|_{L_2(\Omega^-_R)}$,
  with a fixed large $R$. 
  Furthermore we define the closed subspace $\mH^2_{k_-}\subset\mH_3$, consisting of $f= f_{1\ta}+ f_{1\no}+ f_{2\ta}+ f_{2\no}$,
  with $f_{1\ta}, f_{1\no}$ being tangential and normal vector fields
  and $f_{2\ta}, f_{2\no}$ being tangential and normal bivector fields,
  satisfying the constraints $d_\ta f_{1\ta}= ik_- f_{2\ta}$ and
  $\del_\ta(\nu\lctr f_{2\no})= -ik_- \nu\lctr f_{1\no}$.
  The traces of incoming Maxwell fields $F^0$ all belong to $\mH^2_{k_-}$.
  
The transmission problem \eqref{eq:maxwtranspr} defines a bounded 
linear map
\begin{equation}
   T_{\hat\epsilon, k_+, k_-}: \widetilde \mH_{k_+}^{2,+}
   \oplus \widetilde \mH_{k_-}^{2,-}\to \mH^2_{k_-}.
\end{equation}
Like in  Proposition~\ref{prop:helmexist}, we can continuously perturb
the operator and spaces to the case  $k_+=k_-$ and that $\hat\epsilon=1$,
and, given (ii) below, conclude by Fredholm perturbation theory that 
$T_{\hat\epsilon, k_+, k_-}$ is a Fredholm operator with index zero
and hence an isomorphism whenever \eqref{eq:maxwcondforexist} holds.

(ii) 
To establish the estimate  of  \eqref{eq:maxwneededest}, in
$\Omega_R^-$ we construct potentials such that
\begin{equation}
\begin{cases}
  F^-_2= \nabla\wedg U^-_1,\\
  F^-_1= \nabla\wedg U^-_0+ ik_- U^-_1,
\end{cases}
\end{equation}
and in $\Omega^+$ we construct potentials such that
\begin{equation}
\begin{cases}
  F^+_2= \nabla\lctr U^+_3+ ik_+ U^+_2,\\
  F^+_1= \nabla\lctr U^+_2,
\end{cases}
\end{equation}
where subscript $j$ refers to a $\wedge^j \R^3$ valued field.
 (In traditional terminology, $U^-_0$ is the scalar potential and $U^-_1$
is the vector potential to the electromagnetic field $F^-$.) 
 The existence and estimates of such potentials 
follow from the Hodge decompositions in \cite{AxMcIntosh:04},
by writing out the homogeneous parts of the multivector fields. 

We consider first $F^\pm_1$, where we pair the jump relation for 
$\nu\wedg f_1^\pm$ with $u_2^+$. From the jump relation for $\nu\lctr
f_2^\pm$ and Maxwell's equations, we obtain
\begin{equation}  \label{eq:gaussjump1}
\nu\lctr f^+_1=\hat \epsilon^{-1} \nu\lctr (f^-_1+f^0_1),
\end{equation}
which we pair with $u_0^-$. We get
\begin{equation}    \label{eq:maxexest1}
\begin{cases}
\int_{\bdy \Omega}\scl{\nu\wedg f_1^+}{u_2^+} d\sigma=
\int_{\bdy \Omega}(\scl{\nu\wedg f_1^-}{u_2^+}d\sigma+
\scl{\nu\wedg f_1^0}{u_2^+})d\sigma,\\
\int_{\bdy \Omega}\scl{u_0^-}{\nu\lctr f_1^+} d\sigma
=\rev{1/\hat \epsilon} \int_{\bdy \Omega}(\scl{u_0^-}{\nu\lctr f_1^-}+
\scl{u_0^-}{\nu\lctr f_1^0}) d\sigma.
\end{cases}
\end{equation}

Next consider $F^\pm_2$, where we pair the jump relation for $\nu\lctr
f_2^\pm$ with $u_1^-$. From the jump relation for $\nu\wedg f_1^\pm$
and Maxwell's equations, we obtain
\begin{equation}  \label{eq:gaussjump2}
\nu\wedg f^+_2=\hat k^{-1} \nu\wedg (f^-_2+f^0_2),
\end{equation}
which we pair with $u_3^+$. We get
\begin{equation}  \label{eq:maxexest2}
\begin{cases}
\int_{\bdy \Omega}\scl{\nu\wedg f_2^+}{u_3^+} d\sigma=(1/\hat k)
\int_{\bdy \Omega}(\scl{\nu\wedg f_2^-}{u_3^+}+
\scl{\nu\wedg f_2^0}{u_3^+})d\sigma,\\
\int_{\bdy \Omega}\scl{u_1^-}{\nu\lctr f_2^+} d\sigma
=\rev{\hat k/\hat \epsilon} \int_{\bdy \Omega}(\scl{u_1^-}{\nu\lctr f_2^-}+
\scl{u_1^-}{\nu\lctr f_2^0}) d\sigma.
\end{cases}
\end{equation}

Applying the general Stokes theorem, see \cite[Sec. 7.3]{RosenGMA:19},
to \eqref{eq:maxexest1} and \eqref{eq:maxexest2} and summing the so
obtained estimates yield an estimate of \eqref{eq:maxwneededest},
similar to the proof of Proposition~\ref{prop:helmexist}. A main idea
is that the potentials $U^\pm_j$ depend compactly on the fields
$F^\pm_j$, and for details of the estimates we refer to \cite[Lem.
4.9, 4.17]{AxThesisPub4:06}.
\end{proof}

With this solvability result for Maxwell's equations, we next derive 
solvability results for the 3D Dirac equation similarly to what was done
in 2D in Section~\ref{sec:diracint2d}.

\begin{prop}  \label{prop:diractrans3d}
  Consider the Dirac transmission problem \eqref{eq:diractranspr}  for 
  a bounded Lipschitz domain $\Omega^+\subset \R^3$ with connected 
  exterior $\Omega^-$. 
Let
\begin{equation}   \label{eq:dirjump3dM}
  M= \diag\begin{bmatrix} \hat k/(\alpha\beta) & 1/\hat k & \hat k/\alpha & \hat k/\alpha  & 1/\gamma & 1/\alpha 
   & 1& 1  \end{bmatrix},
\end{equation}
with parameters $\alpha,\beta,\gamma\in\C\setminus\{0\}$.
Assume that 
\begin{equation}
  \alpha,  \hat k^2/\alpha,  \beta,\gamma\in \C\setminus [-C(\bdy \Omega), -1/C(\bdy \Omega)] \quad\text{and}\quad \alpha/\hat k, \beta/\hat k, \gamma/\hat k \in \swp(k_-,k_+),
\end{equation} 
recalling that $\hat k= k_+/k_-$ and Definition~\ref{defn:hexagon}. 
Then the operator $B_3:E_{k_+}^+\mH_3\oplus E_{k_-}^-\mH_3 \to \mH_3$ given by
\begin{equation}
  B_3(f^+,f^-):= f^+-Mf^-
\end{equation}
is invertible,  where we as in Section~\ref{sec:tracespace} 
denote the spaces of traces
of solutions on $\bdy \Omega$ by $E_{k_\pm}^\pm\mH_3$. 
\end{prop}

\begin{proof}
The proof is similar to that of Proposition~\ref{prop:diractrans},
but now using that in 3D the Dirac equation contains two Helmholtz equations
along with Maxwell's equations.
We write 
\begin{equation}
   F^\pm= F^\pm_0 + F^\pm_1+F^\pm_2+ F^\pm_3
\end{equation}
and proceed in three steps. We first prove uniqueness.

(i)
The scalar functions $F_0^\pm$ solve the Helmholtz equation as a
consequence of $F^\pm$ solving the Dirac equation, which also shows
\begin{equation} 
  \nabla F^\pm_0= ik_\pm F^\pm_1-\nabla\lctr F^\pm_2.
\end{equation}
From this and \eqref{eq:dirjump3dM}, we conclude that $\pd_\nu
f_0^+=\beta \pd_\nu(\hat k f^-_0/(\alpha\beta))$ and $f_0^+= \hat k
f^-_0/(\alpha\beta)$ and hence Proposition~\ref{prop:helminj} with
$\hat \epsilon= \beta$ shows that $F_0^\pm=0$.

(ii)
Writing $F^\pm_3= U^\pm_3 e_{123}$, as in (i) the scalar functions
$U^\pm_3$ solve the Helmholtz equation. From the Dirac equation we
have
\begin{equation}  
  \nabla\lctr F^\pm_3= ik_\pm F^\pm_2-\nabla\wedg F^\pm_1.
\end{equation}
From this and \eqref{eq:dirjump3dM}, we conclude that 
$\pd_\nu u_3^+=\gamma \pd_\nu( u^-_3/\gamma)$ and 
$u_3^+= u^-_3/\gamma$ 
and hence Proposition~\ref{prop:helminj} with 
$\hat \epsilon= \gamma$ shows that $F_3^\pm=0=U^\pm_3$.

(iii)
From (i) and (ii) we conclude that $F^\pm$ solve Maxwell's equations
\eqref{eq:maxwtranspr} with $\hat\epsilon=\alpha$, and
Proposition~\ref{prop:maxwinj} applies to show that $F^\pm=0$.

To show existence, it suffices by perturbation theory for Fredholm operators
to prove an estimate
\begin{equation} 
  \|f^+\|_{\mH_3}+ \|f^-\|_{\mH_3}\lesssim \|f^0\|_{\mH_3}.
\end{equation} 
This follows similarly to steps (i)--(iii) by instead appealing to
Propositions~\ref{prop:helmexist} and \ref{prop:maxwexist}.
\end{proof}

\begin{prop}   \label{prop:diracintwp3d}
  Assume  the hypothesis of Proposition~\ref{prop:diractrans3d}, 
and further assume that 
\begin{equation}
  M'= \diag\begin{bmatrix} 1/\alpha' & 1/\gamma' & 1 & 1 & \hat k & 1/(\hat k\alpha'\beta') &
  1/(\alpha'\hat k) & 1/(\alpha'\hat k)  \end{bmatrix},
\end{equation}
with parameters satisfying 
\begin{equation}
  \alpha',  \alpha' \hat k^2,  \beta',\gamma'\in \C\setminus [-C(\bdy \Omega), -1/C(\bdy \Omega)] \quad\text{and}\quad \alpha'\hat k, \beta'\hat k, \gamma' \hat k\in \swp(k_+,k_-).
\end{equation}
Then the operator 
\begin{equation}   \label{eq:predirint3d}
  rE_{k_+}^+M'+M E_{k_-}^-
\end{equation}
is invertible on $\mH_3$, for any $r\in\C\setminus\{0\}$.
\end{prop}

\begin{proof}
  This follows by duality from Proposition~\ref{prop:diractrans3d}, 
  entirely analogously 
  to the proof of Proposition~\ref{prop:diracintwp}.
\end{proof}

Consider the 3D Dirac integral equation 
\begin{equation}    \label{eq:3dgendirint}
  (rE_{k_+}^+M'+M E_{k_-}^-)\tilde h= Mf^0,
\end{equation}
involving the operator from \eqref{eq:predirint3d}, 
with right-hand side specified by the jump condition in \eqref{eq:diractranspr}
and auxiliary density $\tilde h$ which we precondition as $\tilde h= P'h$ in
the proof of Theorem~\ref{thm:3Ddiracint}. 
Similar to what we did for the 2D twin  \eqref{eq:gendirint}, we optimize
\eqref{eq:3dgendirint} by 
choosing the parameters $r,\beta,\gamma,\alpha',\beta',\gamma'$.
Recall that for EM fields $\alpha=\hat\epsilon$.
We therefore consider $\alpha$ as having a prescribed value.
Writing \eqref{eq:predirint3d} as in \eqref{eq:reflopexp}, we have in $\R^3$
that
\begin{equation}
  rM'+M= \diag\begin{bmatrix}\tfrac r{\alpha'}+\tfrac{\hat k}{\alpha\beta} & \tfrac r{\gamma'}+\tfrac 1{\hat k} & r+\tfrac{\hat k}\alpha & r+\tfrac{\hat k}\alpha  & r\hat k+\tfrac 1\gamma
  & \tfrac {r}{\hat k\alpha'\beta'}+\tfrac 1 \alpha &
  \tfrac {r}{\alpha'\hat k}+1 & \tfrac {r}{\alpha'\hat k}+1  \end{bmatrix}.
\end{equation}
 Modulo operators of the form  $K^v_{k_\pm}-K^v$,
$S^a_{k_+}-ik_-\hat kS^a$ and $S^a_{k_-}-ik_-S^a$,  the integral operator $T=rE_{k_+} M'-M E_{k_-}$ from \eqref{eq:reflopexp}
is the entry-wise product of  
\begin{equation}    \label{eq:prediracintop3d}
  \begin{bmatrix}
\tfrac r{\alpha'}-\tfrac{\hat k}{\alpha\beta} & \tfrac r{\gamma'}-\tfrac{\hat k}{\alpha\beta} & r-\tfrac{\hat k}{\alpha\beta} & r-\tfrac{\hat k}{\alpha\beta}  & r\hat k^2-\tfrac{\hat k}{\alpha\beta} & \tfrac {r}{\alpha'\beta'}-\tfrac{\hat k}{\alpha\beta} &
  \tfrac {r}{\alpha'}-\tfrac{\hat k}{\alpha\beta} & \tfrac {r}{\alpha'}-\tfrac{\hat k}{\alpha\beta} \\
  \tfrac r{\alpha'}-\tfrac 1{\hat k} & \tfrac r{\gamma'}-\tfrac 1{\hat k} & r-\tfrac 1{\hat k} & r-\tfrac 1{\hat k} & r\hat k
^2-\tfrac 1{\hat k} & \tfrac {r}{\alpha'\beta'}-\tfrac 1{\hat k} &
  \tfrac {r}{\alpha'}-\tfrac 1{\hat k} & \tfrac {r}{\alpha'}-\tfrac 1{\hat k}  \\
\tfrac r{\alpha'}-\tfrac{\hat k}\alpha & \tfrac r{\gamma'}-\tfrac{\hat k}\alpha & r-\tfrac{\hat k}\alpha & r-\tfrac{\hat k}\alpha & r\hat k^2-\tfrac{\hat k}\alpha & \tfrac {r}{\alpha'\beta'}-\tfrac{\hat k}\alpha &
  \tfrac {r}{\alpha'}-\tfrac{\hat k}\alpha & \tfrac {r}{\alpha'}-\tfrac{\hat k}\alpha  \\
 \tfrac  r{\alpha'}-\tfrac{\hat k}\alpha & \tfrac r{\gamma'}-\tfrac{\hat k}\alpha & r-\tfrac{\hat k}\alpha & r-\tfrac{\hat k}\alpha & r\hat k^2-\tfrac{\hat k}\alpha & \tfrac {r}{\alpha'\beta'}-\tfrac{\hat k}\alpha &
  \tfrac {r}{\alpha'}-\tfrac{\hat k}\alpha & \tfrac {r}{\alpha'}-\tfrac{\hat k}\alpha  \\
  \tfrac {r\hat k}{\alpha'}-\tfrac 1\gamma & \tfrac {r\hat k}{\gamma'}-\tfrac 1\gamma & r\hat k-\tfrac 1\gamma & r\hat k-\tfrac 1\gamma & r\hat k-\tfrac 1\gamma & \tfrac {r}{\hat k\alpha'\beta'}-\tfrac 1\gamma &
  \tfrac {r}{\alpha'\hat k}-\tfrac 1\gamma & \tfrac {r}{\alpha'\hat k}-\tfrac 1\gamma  \\
  \tfrac {r\hat k}{\alpha'}-\tfrac 1 \alpha & \tfrac {r\hat k}{\gamma'}-\tfrac 1 \alpha & r\hat k-\tfrac 1 \alpha & r\hat k-\tfrac 1 \alpha & r\hat k-\tfrac 1 \alpha & \tfrac {r}{\hat k\alpha'\beta'}-\tfrac 1 \alpha &
  \tfrac {r}{\alpha'\hat k}-\tfrac 1 \alpha & \tfrac {r}{\alpha'\hat k}-\tfrac 1 \alpha  \\
  \tfrac {r\hat k}{\alpha'}-1 & \tfrac {r\hat k}{\gamma'}-1 & r\hat k-1 & r\hat k-1 & r\hat k-1 & \tfrac {r}{\hat k\alpha'\beta'}-1 &
  \tfrac {r}{\alpha'\hat k}-1 & \tfrac {r}{\alpha'\hat k}-1  \\
 \tfrac  {r\hat k}{\alpha'}-1 & \tfrac {r\hat k}{\gamma'}-1 & r\hat k-1 & r\hat k-1 & r\hat k-1  & \tfrac {r}{\hat k\alpha'\beta'}-1 &
  \tfrac {r}{\alpha'\hat k}-1 & \tfrac {r}{\alpha'\hat k}-1 
  \end{bmatrix}.
\end{equation}
and $E_k$, 
with diagonal and off-diagonal $2\times 2$ blocks replaced by $K^v$ and
$ik_- S^a$ respectively. 
It is the diagonal $4\times 4$ blocks in \eqref{eq:prediracintop3d}
which are our main concern, within which we note that the diagonal
$2\times 2$ blocks are weakly singular operators on smooth domains.

\begin{itemize}
\item Our first choice is to set $r=1/\hat k$,  $\alpha\beta=\hat k^2$ and
  $\alpha'\beta'=1/\hat k^2$.  This gives cancellation in the (1:2,3:4) and
  (7:8,5:6) size $2\times 2$ blocks of $T$, which for a smooth domain
  $\Omega$ yields  a nilpotent operator $T$ with essential spectrum 
  $\{0\}$, if $\mH_3$ is replaced by a function space of fixed 
  regularity. 
  
\item It remains to choose $\gamma, \alpha',\gamma'$. The choice of
  $\alpha'$ concerns the diagonal elements (1,1), (7,7) and (8,8). We
  set $\alpha'= 1/\hat k$, and so $\beta'=1/\hat k$. The choices of
  $\gamma, \gamma'$ concern the diagonal elements (5,5) and (2,2). We
  set $\gamma= \hat k/|\hat k|$ and $\gamma'= \rev{\hat k}/|\hat k|$.
  
  These choices make $\beta/\hat k=(\alpha/\hat k)^{-1}$, $\alpha'\hat
  k=\beta' \hat k=1$, $\gamma' \hat k>0$ and $\gamma/\hat k>0$.
  Therefore, if $\alpha, \hat k^2/\alpha\in \C\setminus (-\infty,0]$
  and $\alpha/\hat k\in\swp(k_-,k_+)$, then
  Propositions~\ref{prop:diractrans3d} and \ref{prop:diracintwp3d}
  guarantee invertibility of $rE_{k_+}^+M'+M E_{k_-}^-$. Note that for
  non-magnetic materials $\hat k^2/\alpha=\beta=\hat \mu=1$. Similar
  to the situation in $\R^2$, we also obtain good invertibility
  properties of the (1,1), (2,2), (5,5) diagonal
  elements,  as well as the (7:8,7:8) diagonal block,  
  on any Lipschitz domain in this way.
  Furthermore, with $\hat\mu=1$ we also obtain good invertibility
  properties for the (3,3) and (4,4) diagonal elements, so that it is
  only the (6,6) element which we do not control 
   in the sense that we may hit its essential spectrum for
  $\hat\epsilon <0$. 
\end{itemize}

To summarize, for solving the Maxwell transmission problem 
\eqref{eq:maxwtransprclassi}, we have obtained the 3D Dirac integral equation
\begin{equation}  \label{eq:3DfinalDirac}
  (\hat k^{-1}M'+M+  E_{k_+} (\hat k^{-1}M')-M E_{k_-})\tilde h= 2M f^0,
\end{equation}
with 
\begin{align}
  M&= \diag\begin{bmatrix} 1/\hat k & 1/\hat k & \hat k/\hat\epsilon & \hat k/\hat\epsilon & \rev{\hat k}/|\hat k| & 1/\hat\epsilon & 1 & 1 \end{bmatrix},\\
  \hat k^{-1}M'&= \diag\begin{bmatrix} 1 & 1/|\hat k| & 1/\hat k & 1/\hat k & 1 & 1 & 1/\hat k & 1/\hat k \end{bmatrix}.
\end{align}

\section{Numerical results for the 2D Dirac integral equation}
\label{sec:numerical}

This section shows how the 2D Dirac integral equation
\eqref{eq:2DDiracequ}, along with the field representation formula
\eqref{eq:ufield2d}, performs numerically when applied to the planar
Helmholtz/TM Maxwell transmission problem. For comparison, we also
investigate the performance of the $4\times 4$ system of integral
equations \cite[Eq.~(12.4)]{HelsingKarlsson:20}, along with its
field representation formulas \cite[Eqs.~(12.2) and~(12.3)]{HelsingKarlsson:20}. For simplicity, we refer to the
system~\eqref{eq:2DDiracequ} as ``Dirac'' and to the
system~\cite[Eq.~(12.4)]{HelsingKarlsson:20} as ``HK 4-dens''.
The reason for comparing with ``HK 4-dens'' is that this system, just
like ``Dirac'', has a $8\times 8$ counterpart in 3D which applies to
Maxwell's equations.

We also compare to a state-of-the-art $2\times 2$ system of integral
equations~\cite[Eq.~(12.7)]{HelsingKarlsson:20} of
Kleinman--Martin type~\cite{KleiMart88} and to the 2D version of the
classic M{\"u}ller system \cite[p. 319]{Muller69}, which also is a
$2\times 2$ system~\cite[Sec.~14.1]{HelsingKarlsson:20}, and
refer to these systems as ``best KM-type'' and ``2D M{\"u}ller''.
While ``best KM-type'' is limited to planar problems it yields, in
general, the best results. For interior wavenumbers $\arg(k_+)=0$,
``best KM-type'' coincides with ``2D M{\"u}ller''.

 We stress, at this point, that the purpose of the numerical tests
is to verify that ``Dirac'' has the properties claimed in the
theoretical sections of this paper. We perform these tests in 2D
simply because we have not yet access to a high-order accurate solver
for ``Dirac'' in 3D. For example, we monitor condition numbers under
wavenumber sweeps in order to detect eigenwavenumbers. For scattering
problems, we investigate achievable field accuracy and the convergence
of iterative solvers. We are not trying to show that ``Dirac'' is more
efficient than ``best KM-type'' in 2D, because it is not. Rather, we
demonstrate that ``Dirac'' is almost as efficient as ``best KM-type''
in 2D and often more efficient than ``HK 4-dens''.  This is of
interest because ``Dirac'' is also applicable in 3D, while ``best
KM-type'' is not.

All systems of integral equations are discretized using Nystr{\"o}m
discretization and underlying composite $16$-point Gauss--Legendre
quadrature.  The reason for using Nystr{\"o}m discretization,
rather than the more common Galerkin discretization (BEM or MoM), has
to do with accuracy and speed. High achievable accuracy and fast
execution is, in our opinion, more easily obtained in a Nystr{\"o}m
setting. This applies not only in 2D, but also to scattering problems
involving rotationally symmetric interfaces $\partial\Omega$ in 3D.
See, for example,~\cite[Sec.~VII.A]{HelsingKarlsson:17}, where
double-precision results, obtained with a high-order
Fourier--Nystr{\"o}m method, agree to almost $14$ digits with
semi-analytical solutions given by Mie theory for an electromagnetic
transmission eigenproblem on the unit sphere. See
also~\cite{EpsteinGreengardONeil:19,LaiOneil:19} for other recent
examples where authors prefer Nystr{\"o}m schemes over Galerkin for
electromagnetic scattering. The discussion of Nystr{\"o}m versus
Galerkin discretization of integral equations in computational
electromagnetics in~\cite[Sec.~1]{LaiOneil:19} is very informative.

On smooth $\bdy\Omega$ we use a variant of the ``Nystr{\"o}m scheme
B'' in~\cite{HelsHols15}. In the presense of singular boundary points
on $\bdy\Omega$ such as corner vertices, which would cause performance
degradation in a naive implementation, the Nystr{\"o}m scheme is
stabilized and accelerated using recursively compressed inverse
preconditioning (RCIP)~\cite{Hels18}.  We note, in passing, that
RCIP is applicable also to Fourier--Nystr{\"o}m discretization near
sharp edges and singular boundary points on rotationally symmetric
$\partial\Omega$ in 3D~\cite{HelsingKarlsson:16,HelsingPerfekt:18}.
 Accurate evaluation of singular operators on $\bdy\Omega$ and
accurate field evaluation of layer potentials at field points close to
$\bdy\Omega$ are accomplished using panel-wise product integration.
See~\cite[Sec.~4]{HelsingKarlsson:18}, and references therein, for
details.  See also~\cite[Sec.~9.1]{HelsingKarlsson:18} for a
thorough test of the implementation of the integral operators needed
in ``Best KM-type'' and ``2D M{\"u}ller'' via Calder{\'o}n identites
and see~\cite[Sec.~9]{HelsHols15} for a high-wavenumber test of the
implementation of the corresponding layer potentials needed for field
evaluations.  Large discretized linear systems are solved
iteratively using GMRES ( without restart).

Our codes are implemented in {\sc Matlab}, release 2018b, and executed
on a workstation equipped with an Intel Core i7-3930K CPU and 64 GB of
RAM. Fields are computed at $10^6$ points on a rectangular Cartesian
grid in the computational domains shown. When assessing the accuracy
of computed fields we compare to a reference solution. The reference
solution is either obtained from a system deemed to give more accurate
solutions, or by overresolution using roughly 50\% more points in the
discretization of the system under study. The absolute difference
between the original solution and the reference solution is called the
{\it estimated absolute error}.

\begin{figure}[t]
\centering
  \begin{subfigure}[b]{0.49\linewidth}
     \includegraphics[height=55mm]{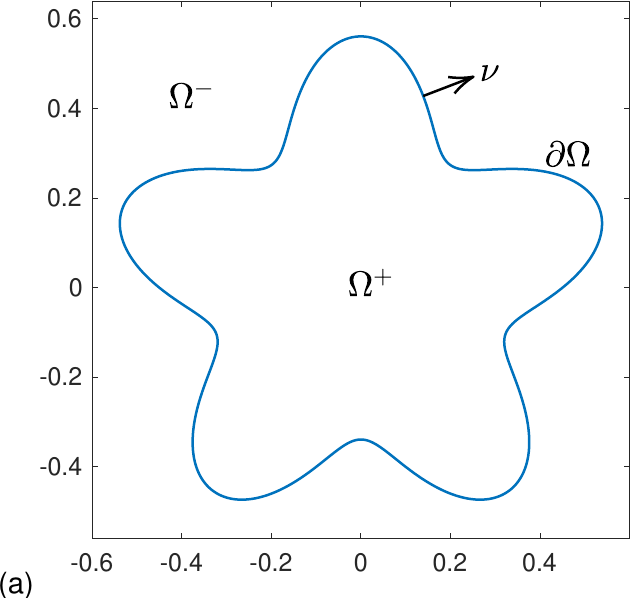}
  \end{subfigure}
  \begin{subfigure}[b]{0.49\linewidth}
     \includegraphics[height=55mm]{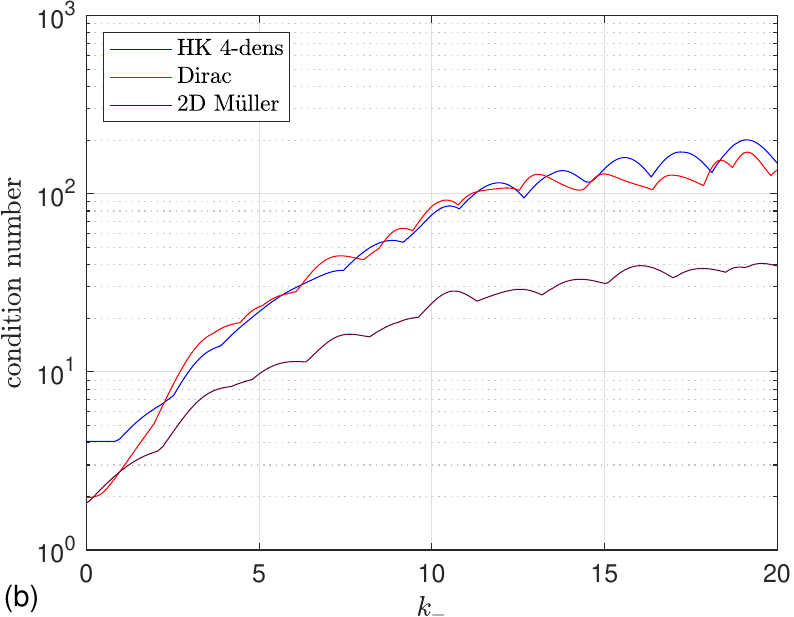}
  \end{subfigure}
\par\bigskip
  \begin{subfigure}[b]{0.49\linewidth}
     \includegraphics[height=55mm]{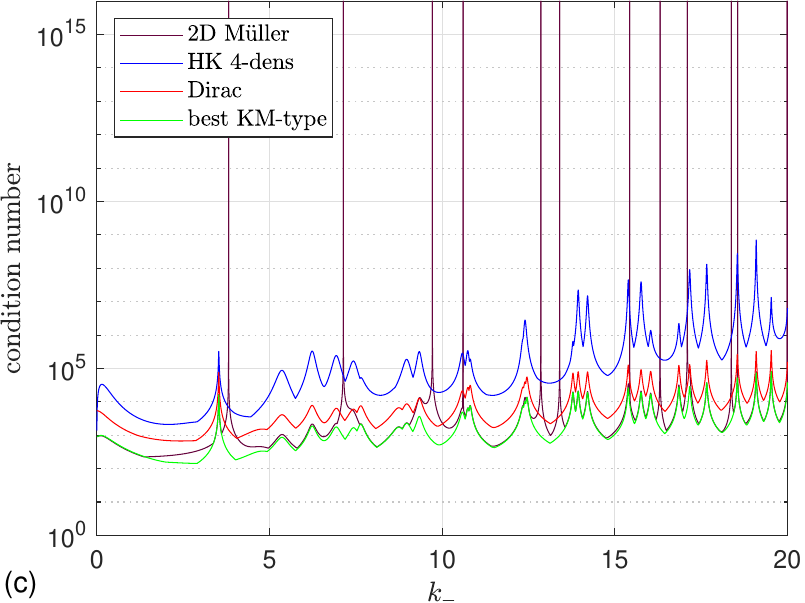}
  \end{subfigure} 
  \begin{subfigure}[b]{0.49\linewidth}
     \includegraphics[height=55mm]{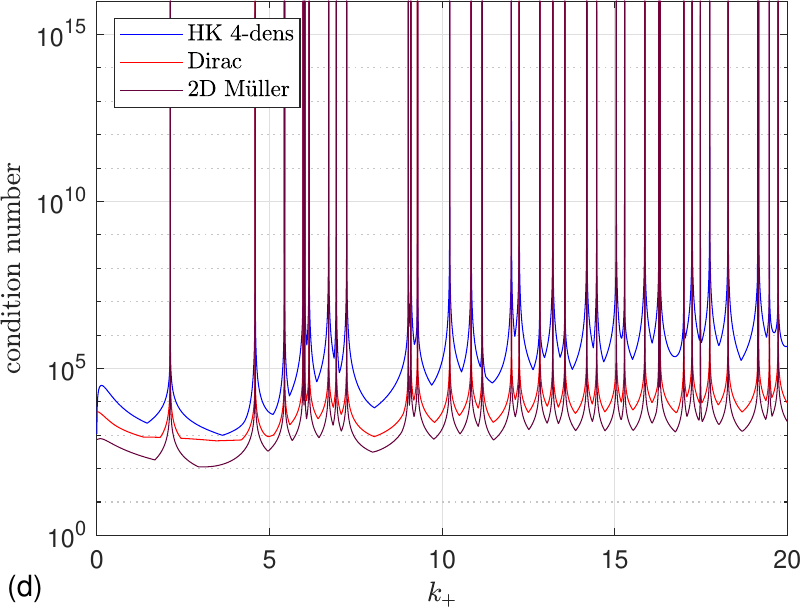}
  \end{subfigure}
\caption{\sf Condition numbers of the operators in ``Dirac'', in 
  ``HK 4-dens'', in ``best KM-type'', and in ``2D M{\"uller}'': (a) the
  starfish-like interface $\bdy\Omega$; (b) the positive dielectric
  case; (c) the plasmonic case; (d) the reverse plasmonic case.}
\label{fig:starfishcondition}
\end{figure}

\subsection{The operators}   \label{subseq:ops}

We compute condition numbers of the discretized system matrices in the
systems under study.  The main  purpose is to detect false
eigenwavenumbers. Another purpose is to compare the conditions number
of the system matrices with each other. The interface $\bdy\Omega$ is
the smooth starfish-like curve \cite[eq.~(92)]{HelsingKarlsson:18},
shown in Figure~\ref{fig:starfishcondition}(a) and originally
suggested for scattering problems in \cite{HaBaMaYo14}. A number of
$976$ discretization points are placed on $\bdy\Omega$. We study three
cases of jumps $\hat k$ where we vary $k_-$ or $k_+$:
\begin{itemize}
\item {\it The positive dielectric case}. The exterior wavenumber is
  positive real, $0<k_-\le 20$, and $\hat k=k_+/k_-= 1.5$ so that
  $0<k_+\le 30$ and $\hat\epsilon=2.25$. This corresponds to the lower
  left corner point in Figure~\ref{fig:hexagon}, where
  Theorem~\ref{thm:2Ddiracint} guarantees that no true
  eigenwavenumbers exist, as well as no false eigenwavenumbers for
  ``Dirac''.
  
  Figure~\ref{fig:starfishcondition}(b) shows that ``Dirac'' and ``HK
  4-dens'' perform equally well, except at low frequencies where the
  condition number of ``Dirac'' fares better and is comparable to that
  of ``2D M{\"u}ller''.
  
\item {\it The plasmonic case}. Again the exterior wavenumber is
  positive real, $0<k_-\le 20$, but $\hat k=k_+/k_-= i\sqrt{1.1838}$
  so that $0<k_+/i\le 20\sqrt{1.1838}$, $k_+$ is imaginary, and
  $\hat\epsilon=-1.1838$ is rather close to the essential spectrum
  $\{-1\}$. This corresponds to the left middle corner point in
  Figure~\ref{fig:hexagon}, where Theorem~\ref{thm:2Ddiracint}
  guarantees that no true eigenwavenumbers exist, as well as no false
  eigenwavenumbers for ``Dirac''.
  
  Figure~\ref{fig:starfishcondition}(c) shows that the condition
  number of ``Dirac'' is closer to that of ``best KM-type'' than to
  ``HK 4-dens'', in particular at high frequencies. ``2D M{\"u}ller''
  exhibits $12$ false eigenwavenumbers.
  
\item {\it The reverse plasmonic case}. Now the interior wavenumber is
  positive real, $0<k_+\le 20$, and $\hat k=k_+/k_-=
  (i\sqrt{1.1838})^{-1}$ so that $0<k_-/i\le 20\sqrt{1.1838}$, $k_-$
  is imaginary, and  $\hat\epsilon=-1/1.1838$.  This corresponds to the
  lower middle corner point in Figure~\ref{fig:hexagon}, which does
  not belong to the hexagon, and indeed
  Figure~\ref{fig:starfishcondition}(d) shows $37$ true
  eigenwavenumbers. The different systems here have a relative
  performance similar to that of the plasmonic case, with ``Dirac''
  performing rather close to ``2D M{\"u}ller''.
\end{itemize}

\subsection{Field computations}  
\label{subsec:fieldplot}

We solve the Dirac system~\eqref{eq:2DDiracequ} for $h$, compute
interior and exterior fields $U^\pm$ via \eqref{eq:ufield2d}, and
compare with results from the other systems. Gradient fields $\nabla
U^\pm$ are not computed, but we remark that the representation formula
\eqref{eq:ugradfield2d} for $\nabla U^\pm$ uses layer potentials with
the same type of (near-logarithmic and near-Cauchy) singular kernels
as the representation formula \eqref{eq:ufield2d}. The $2\times 2$
systems, on the other hand, have accompanying representation formulas
for $\nabla U^\pm$ with layer potentials that contain
near-hypersingular kernels. Evaluation of these layer potentials at
field points near $\bdy\Omega$ may cause additional loss of accuracy.

The curve $\bdy\Omega$ is chosen as the boundary of the one-corner
drop-like object \cite[Eq.~(92)]{HelsingKarlsson:18}, discernible in
Figure~\ref{fig:fieldspositive}(a). We let $k_-=18$ in the positive
dielectric and plasmonic cases, and $k_+=18$ in the reverse plasmonic
case. The incoming wave is a plane wave from south-west, $u^0(x,y)=
e^{ik_-(x+y)/\sqrt 2}$, and $800$ discretization points are placed on
$\bdy\Omega$. When $\hat\epsilon=-1.1838$, then
$\pm(\hat\epsilon+1)/(\hat\epsilon-1)$ is in the essential
$H^{1/2}(\bdy\Omega)$-spectrum of \eqref{eq:doublelayer} and a
homotopy-based numerical procedure is adopted where
$\hat\epsilon=-1.1838$ is approached from above in the complex
plane~\cite[Sec. 6.3]{Hels11}.

\begin{figure}[t]
  \centering 
  \includegraphics[height=50mm]{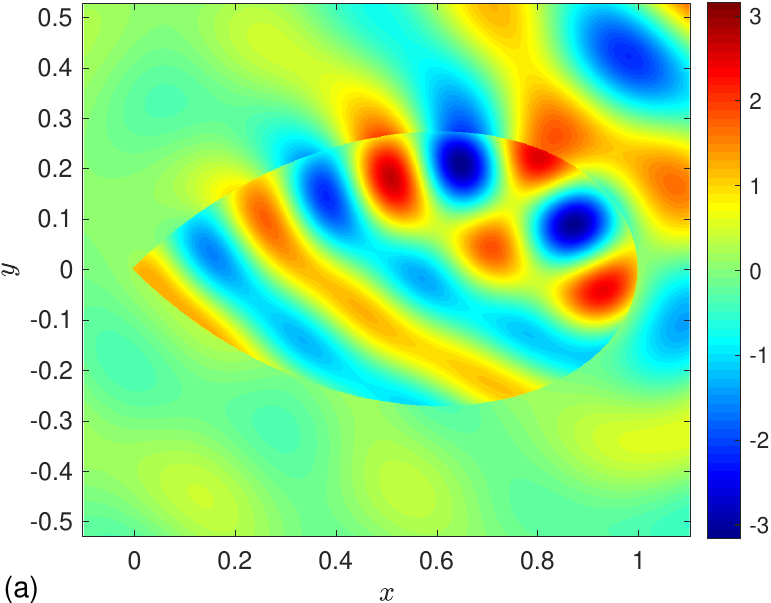}
  \includegraphics[height=50mm]{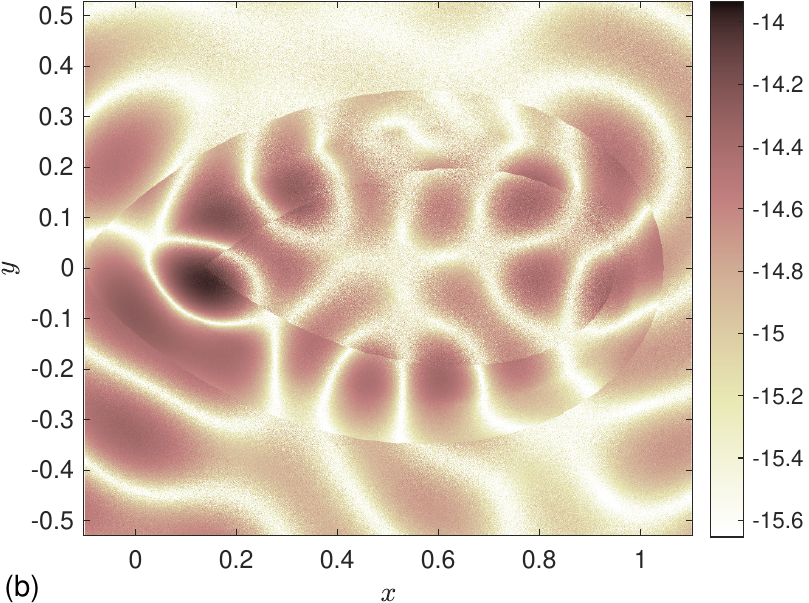}
  \par\bigskip
  \includegraphics[height=50mm]{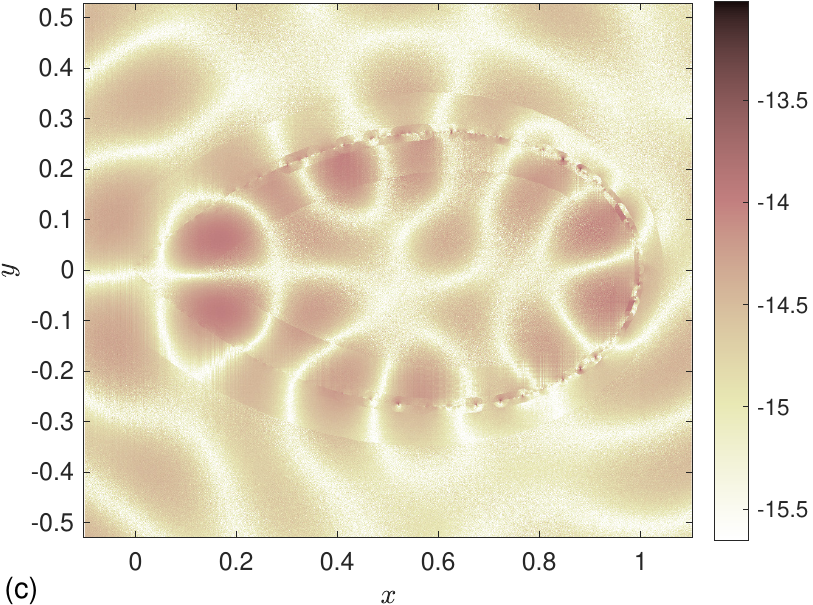}
  \includegraphics[height=50mm]{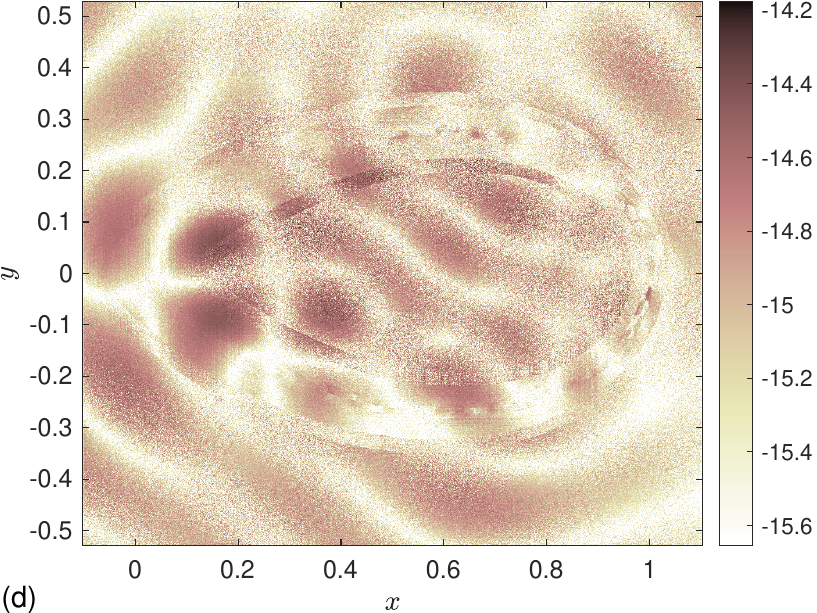}
\caption{\sf Field computation in the positive dielectric case;
  (a) the scattered fields $\re U^\pm$; (b) $\log_{10}$ of the
  estimated absolute error using ``HK 4-dens''; (c) same using
  ``Dirac''; (d) same using ``2D M{\"u}ller''.}
\label{fig:fieldspositive}
\end{figure}

\begin{figure}[t]
  \centering
  \includegraphics[height=50mm]{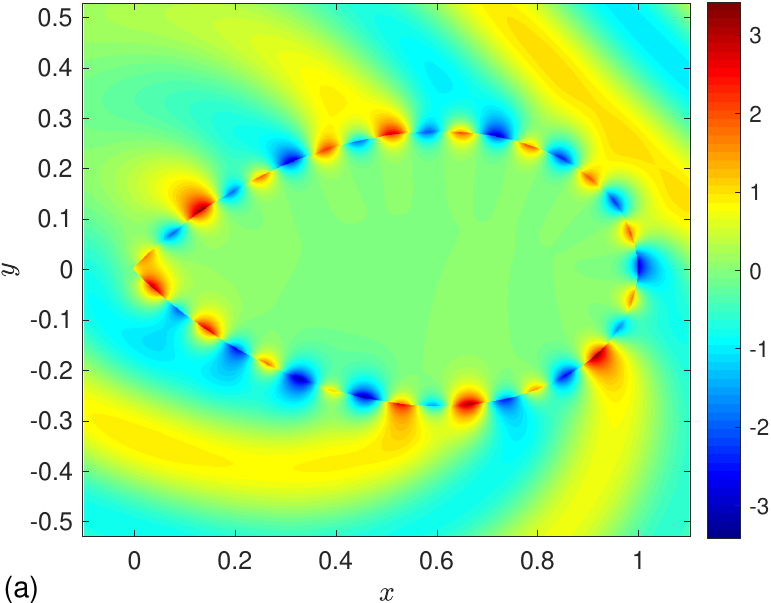}
  \includegraphics[height=50mm]{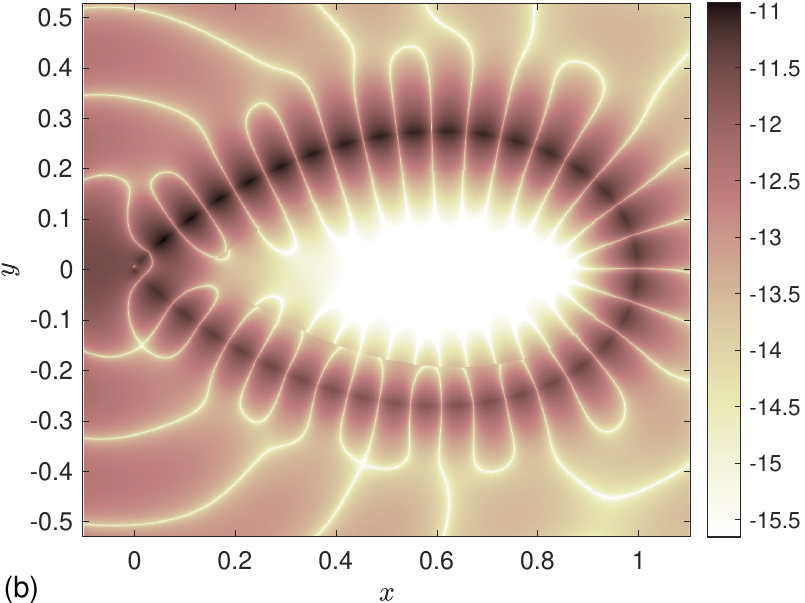}
  \par\bigskip
 \includegraphics[height=50mm]{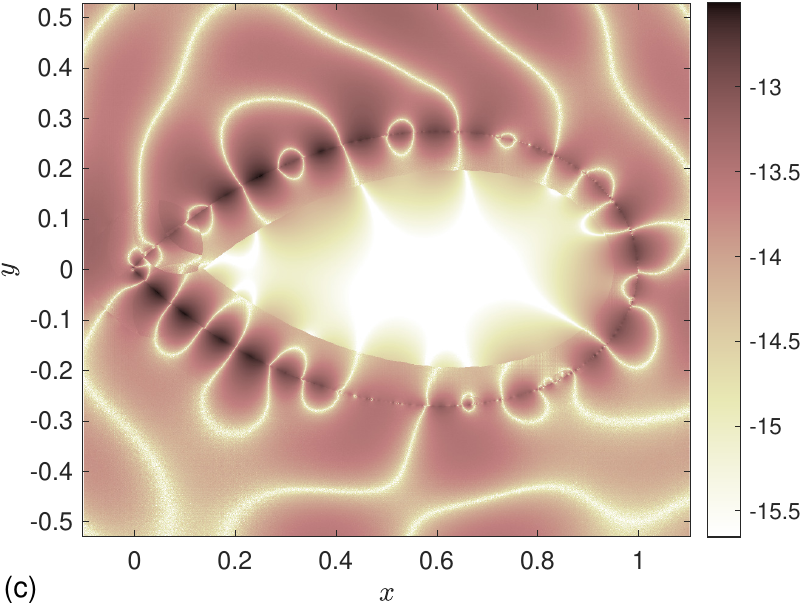}
  \includegraphics[height=50mm]{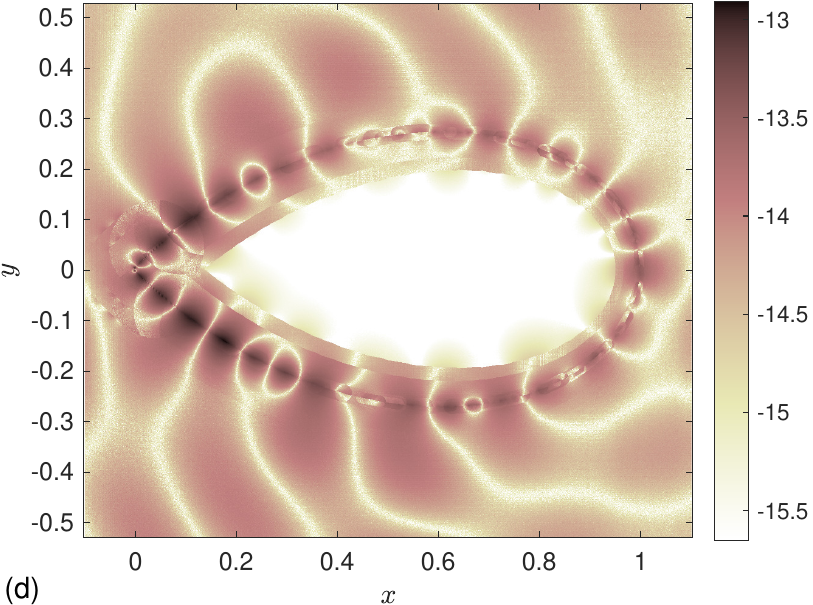}
\caption{\sf Field computation in the plasmonic case;
  (a) the scattered fields $\re U^\pm$; (b) $\log_{10}$ of the
  estimated absolute error using ``HK 4-dens''; (c) same using
  ``Dirac''; (d) same using ``best KM-type''.}
\label{fig:fieldsplasmon}
\end{figure}

\begin{figure}[t]
  \centering
  \includegraphics[height=50mm]{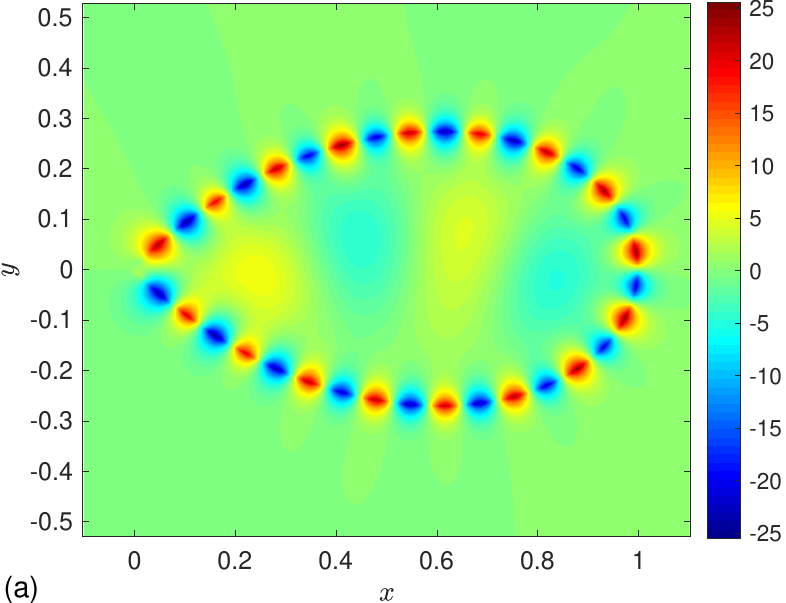}
  \includegraphics[height=50mm]{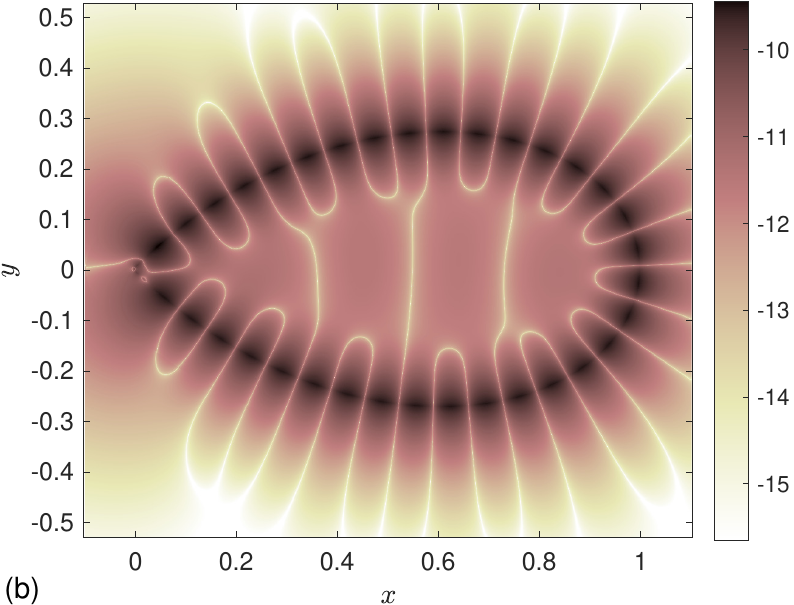}
  \par\bigskip
  \includegraphics[height=50mm]{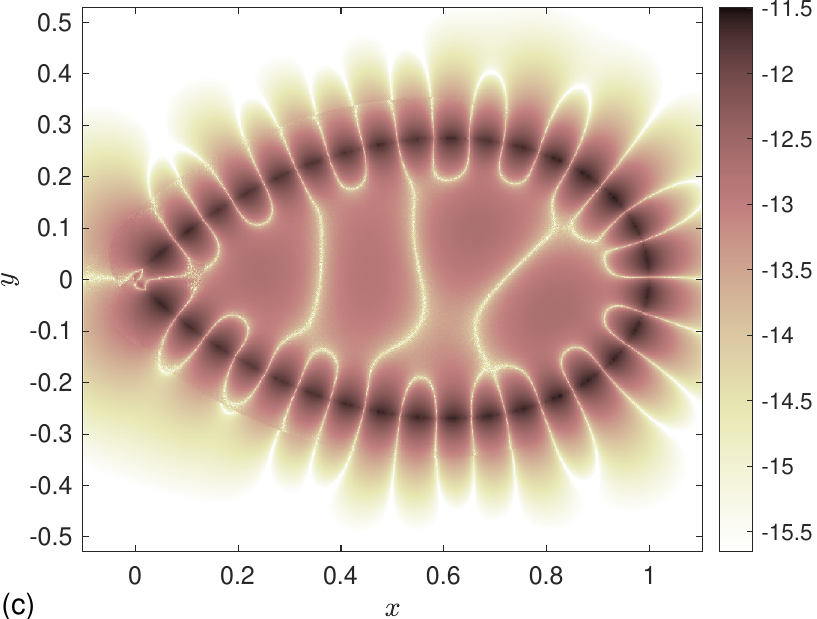}
  \includegraphics[height=50mm]{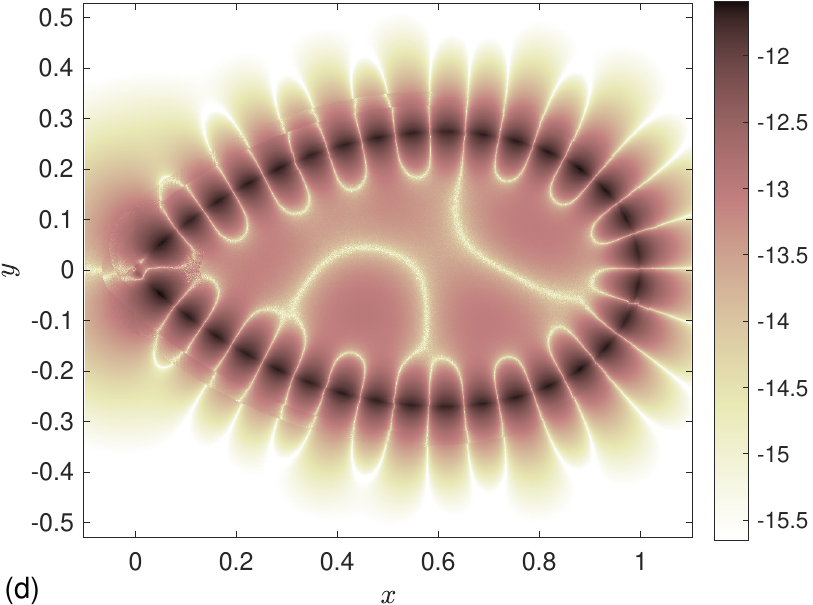}
\caption{\sf Field computation in the reverse plasmonic case;
  (a) the scattered fields $\re U^\pm$; (b) $\log_{10}$ of the
  estimated absolute error using ``HK 4-dens''; (c) same using
  ``Dirac''; (d) same using ``2D M{\"u}ller''.}
\label{fig:fieldsantiplasmon}
\end{figure}

\begin{itemize}
\item Figure~\ref{fig:fieldspositive} covers the positive dielectric
  case. The real parts of the scattered fields $U^\pm$ are shown in
  Figure \ref{fig:fieldspositive}(a). There are propagating waves in
  both $\Omega^+$ and $\Omega^-$. The remaining images show
  $\log_{10}$ of the estimated absolute pointwise error in $\re
  U^\pm$, computed from the different systems. ``Dirac'' loses one
  digit of accuracy in some regions near $\bdy\Omega$ compared to the
  other systems. Most likely, this is because \eqref{eq:ufield2d} does
  not exploit null-fields in the near-field evaluation. See
  \cite[Sec.~7]{HelsingKarlsson:20}.
  
  The number of GMRES iterations needed to meet a stopping criterion
  threshold of machine epsilon in the relative residual are
  $61$, $65$, and $44$ for ``HK 4-dens'', ``Dirac'', and ``2D
  M{\"u}ller'', respectively. In this case the operators in ``HK
  4-dens'' and ``Dirac'' seem to have similar spectral properties,
  while the spectral properties of the operator in ``2D M{\"u}ller''
  are better.
  
\item Figure~\ref{fig:fieldsplasmon} covers the plasmonic case and is
  organized as Figure~\ref{fig:fieldspositive}. There are propagating
  waves in $\Omega^-$, exponentially decaying waves into $\Omega^+$,
  and a surface plasmon wave along $\bdy\Omega$. ``Dirac'' here
  performs almost on par with ``best KM-type'' and gives $2$ more
  accurate digits than ``HK 4-dens''.
  
  The number of GMRES iterations needed are $266$, $173$, and $143$
  for ``HK 4-dens'', ``Dirac'', and ``best KM-type'', respectively.
  In this case the operator in ``Dirac'' seems to have considerably
  better spectral properties than the operator in ``HK 4-dens''.
  
\item Figure~\ref{fig:fieldsantiplasmon} covers the reverse plasmonic
  case. The results are similar to those of the plasmonic case,
  although a digit is lost with all systems and we have propagating
  waves in $\Omega^+$ and exponentially decaying waves into
  $\Omega^-$. ``Dirac'' performs almost on par with ``2D M{\"u}ller''
  and gives $2$ more accurate digits than than ``HK 4-dens''.
%  The number of GMRES iterations needed are $284$, $178$, and $153$
%  for the systems used in Figure~\ref{fig:fieldsantiplasmon}(b),
%  \ref{fig:fieldsantiplasmon}(c) and \ref{fig:fieldsantiplasmon}(d),
%  respectively.
\end{itemize}

\subsection{Densities and function spaces}

We show asymptotics of the density  $h=[h_1\; h_2\; h_3\; h_4]^{\rm T}$ 
obtained from the Dirac integral equation \eqref{eq:2DDiracequ}. The
computations relate to the three examples for the drop-like object in
Section~\ref{subsec:fieldplot}.

\begin{figure}[t!]
  \centering
  \includegraphics[height=58mm]{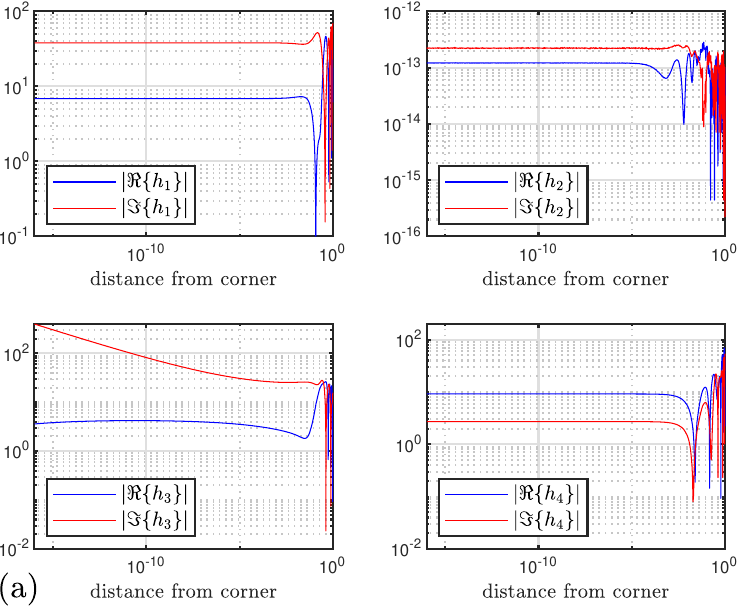}
  \hspace*{5mm}
  \includegraphics[height=58mm]{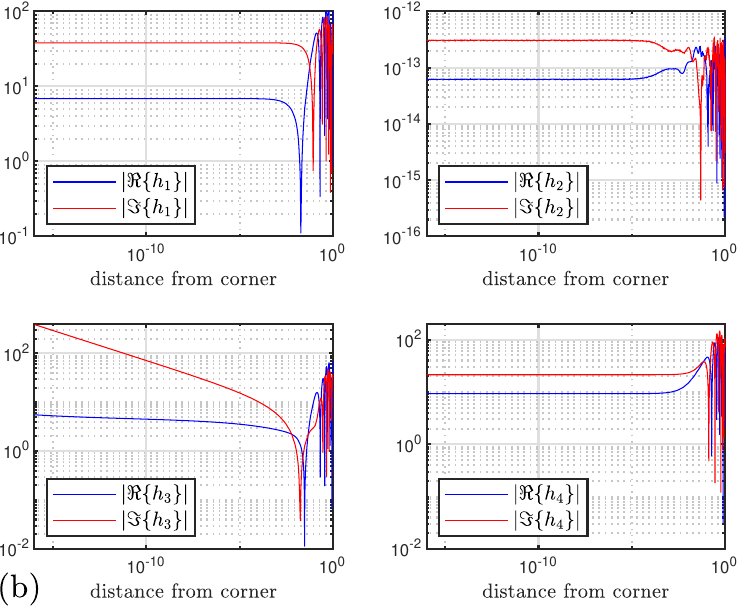}
  \caption{\sf The densities $h_1,h_2,h_3,h_4$ as functions of the 
    arc length distance to the corner vertex in the positive
    dielectric case: (a) along the lower part of $\bdy\Omega$; (b)
    along the upper part of $\bdy\Omega$.}
  \label{fig:densitypositive}
\end{figure}

\begin{figure}[h!]
  \centering
 \includegraphics[height=58mm]{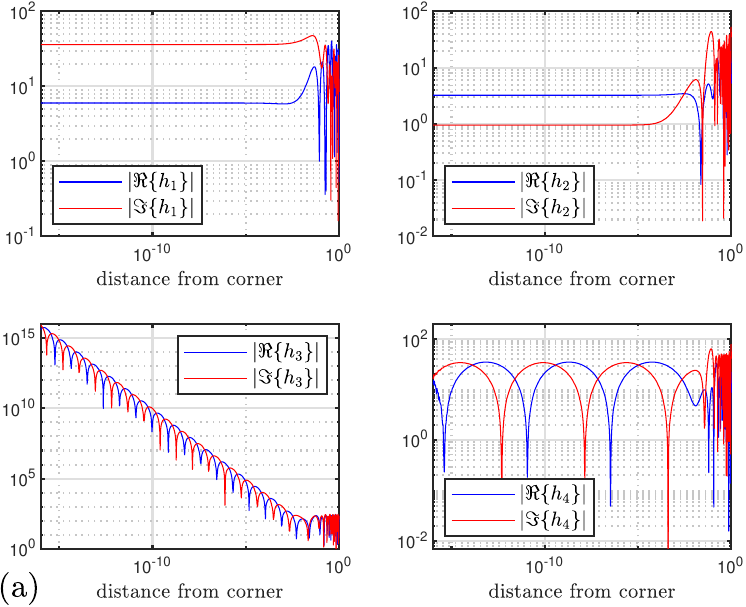}
  \hspace*{5mm}
 \includegraphics[height=58mm]{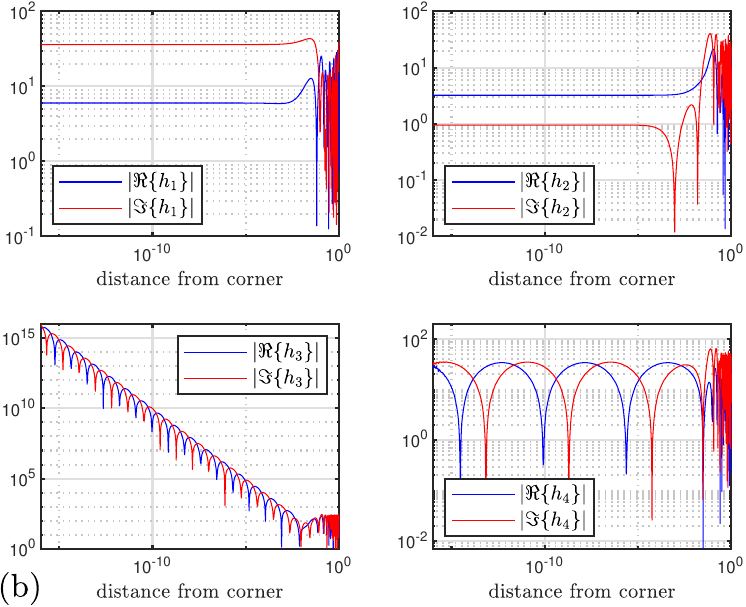}
  \caption{\sf Same as Figure~\ref{fig:densitypositive}, but for
    the plasmonic case.}
  \label{fig:densityplasmon}
\end{figure}

\begin{figure}[h!]
  \centering
 \includegraphics[height=56mm]{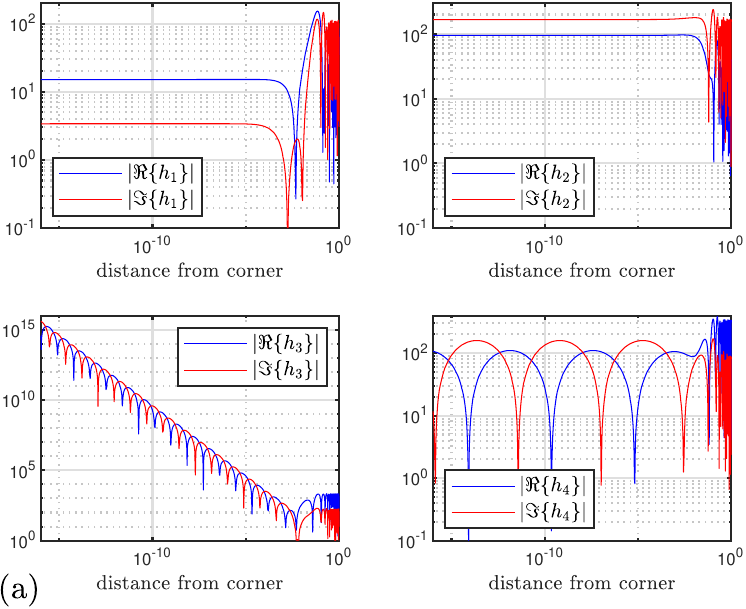}
  \hspace*{5mm}
  \includegraphics[height=56mm]{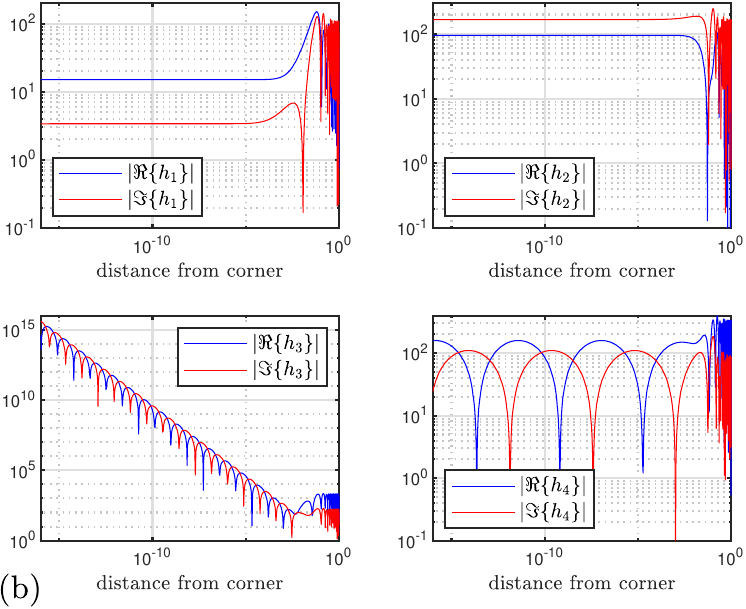}
  \caption{\sf Same as Figure~\ref{fig:densitypositive}, but for
    the reverse plasmonic case.}
  \label{fig:densityantiplasmon}
\end{figure}

\begin{itemize}
\item The positive dielectric case: Here the hypothesis on
  $\hat\epsilon$ and $\hat k$ in Theorem~\ref{thm:2Ddiracint} is
  satisfied. Then $h$ belongs to the energy function space $\mH_2$
  from \eqref{eq:H2space}, meaning that $h_1,h_2\in
  H^{1/2}(\bdy\Omega)$ and $h_3,h_4\in H^{-1/2}(\bdy\Omega)$. The
  result is shown in Figure~\ref{fig:densitypositive}, where indeed
  $h_1$ and $h_2$ are seen to be continuous at the corner vertex (note
  that $h_2\approx 0$). The only singular density is $h_3$, which is
  related to that it is only the $(3,3)$ diagonal element in the
  $4\times 4$ block-operator of \eqref{eq:2DDiracequ} which we do not
  control by the choices of parameters $r,\beta,\alpha',\beta'$ in
  Section~\ref{sec:diracint2d}. Using the automated eigenvalue method
  of~\cite[Sec.~14]{Hels18}, the asymptotic behaviour of $h_3$ near
  the corner is determined to be
  \begin{equation}
  h_3(t)\propto t^\eta,
  \label{eq:asympth}
  \end{equation}
  where $\eta=-0.12319432456634$ and $t$ is the arc length distance
  to the corner vertex. So $h_3$ is in fact in $L_2(\bdy\Omega)$.
  
\item The plasmonic case: Here the hypothesis on $\hat\epsilon$ and
  $\hat k$ in Theorem~\ref{thm:2Ddiracint} is not satisfied since
  $\hat\epsilon<0$ makes $\pm(\hat\epsilon+1)/(\hat\epsilon-1)$ hit
  the essential $H^{-1/2}(\bdy\Omega)$-spectrum of
  \eqref{eq:doublelayer}. Nevertheless, the RCIP-accelerated
  Nystr{\"o}m scheme manages to produce the limit solution $h$ shown
  in Figure~\ref{fig:densityplasmon}. As in the positive dielectric
  case, the densities $h_1,h_2,h_4$ are good, although $h_4$ exhibits
  an oscillatory behaviour. However $h_3\notin H^{-1/2}(\bdy\Omega)$.
  More precisely, its asymptotics near the corner are as
  in~\eqref{eq:asympth} with
  $\eta=-1.00000000000000-i1.57105873276994$. So $h_3\in
  H^{-s}(\bdy\Omega)$ for any $s>1/2$.
  
\item The reverse plasmonic case: the results, shown in
  Figure~\ref{fig:densityantiplasmon}, are very similar to those of
  the plasmonic case. The asymptotics of $h_3$ are as in
  \eqref{eq:asympth} with $\eta=-1.00000000000000+i1.57105873276994$.

\end{itemize}

We end with a remark on the densities $h$ obtained in the plasmonic
and reverse plasmonic cases, which fall outside the energy trace space
$\mH_2$ from \eqref{eq:H2space}. More generally, this energy trace
space belongs to a family of function spaces, where Sobolev regularity
$s=1/2$ and $s-1=-1/2$ is replaced by a more general regularity index
$s$. On Lipschitz domains, the possible range is $0\le s\le 1$. In the
plasmonic cases, our computed densities $h$ belong to the larger
spaces $s<1/2$. For $0\le s<1$, the corresponding norms of the fields
are weighted Sobolev norms using $\int |\nabla U^\pm(x)|^2
\text{dist}(x)^{1-2s}dx$, where $\text{dist}(x)$ denotes distance from
$x$ to $\bdy\Omega$, whereas for the endpoint $s=1$, this must be
replaced by a norm involving a non-tangential maximal function. A
reference for the elementary results for $0<s<1$ is
Costabel~\cite{Costabel:88}. The bounds on double layer potential
operators for the endpoints $s=0,1$ require harmonic analysis and
the Coifman--McIntosh--Meyer theorem~\cite{CMcM:82}.

The essential spectrum of the double layer potential operator
\eqref{eq:doublelayer} depends on the choice of function space, that
is on $s$. For the energy trace space $s=1/2$ this spectrum is a
subset of the real interval $(-1,1)$, but in the endpoint spaces
$s=0,1$, as alluded to in the introduction, this spectrum can be
computed to be a lying ``figure eight'', parametrized as
\begin{equation}   \label{eq:eight}
  \pm\sin(\delta\pi(1+i\xi)/2)/\sin(\pi(1+i\xi)/2), \qquad \xi\in\R,
\end{equation}
where $\delta=\theta/\pi-1$ on a domain with a corner of opening angle
$\theta$.

The point of departure for the investigations reported on in this
paper was the spin integral equation proposed in \cite[Sec.
5]{RosenBoosting:19} for solving the Maxwell transmission problem
\eqref{eq:maxwtransprclassi}. However, it was soon realized that this
was not suitable for the plasmonic and reverse plasmonic cases.
Indeed, the theory developed for  this spin integral equation  makes
use of non-diagonal matrices $P,P',N, N'$ in \eqref{eq:2DDiracequ}
which mix $H^{\pm 1/2}(\bdy \Omega)$ and is limited to the function
space  $L_2(\bdy \Omega)^4$, which is different from $\mH_2$.  
As we have
seen, surface plasmon waves appear in the  function space $\mH_2$
or, in the case of pure meta materials, in
the larger function spaces
$s<1/2$. The numerical algorithm used here
fails for the spin integral equation when
the ``figure eight'' of \eqref{eq:eight} is approached. 
When
$\hat\epsilon= -1.1838$ is approached from above in the complex plane,
this happens near $\hat\epsilon=-1.1838+i0.2168$.

\section*{Acknowledgement}

\noindent
We thank Anders Karlsson for many useful discussions. This work was
supported by the Swedish Research Council under contract
2015-03780.

\bibliographystyle{acm}
%GATHER{AKMcDirac.bib}  % makes sure WinEdt finds citations...
%\bibliography{diracintegral}

\begin{thebibliography}{10}

\bibitem{AxThesisPub2:03}
{\sc Axelsson, A.}
\newblock Oblique and normal transmission problems for {D}irac operators with
  strongly {L}ipschitz interfaces.
\newblock {\em Comm. Partial Differential Equations 28}, 11-12 (2003),
  1911--1941.

\bibitem{AxPhD:03}
{\sc Axelsson, A.}
\newblock {\em Transmission problems for {D}irac's and {M}axwell's equations
  with {L}ipschitz interfaces.}
\newblock PhD thesis, The Australian National University, 2003.
\newblock Available at
  https://openresearch-repository.anu.edu.au/handle/1885/46056.

\bibitem{AxThesisPub1:04}
{\sc Axelsson, A.}
\newblock Transmission problems and boundary operator algebras.
\newblock {\em Integral Equations Operator Theory 50}, 2 (2004), 147--164.

\bibitem{AxThesisPub4:06}
{\sc Axelsson, A.}
\newblock Transmission problems for {M}axwell's equations with weakly
  {L}ipschitz interfaces.
\newblock {\em Math. Methods Appl. Sci. 29}, 6 (2006), 665--714.

\bibitem{AxGrognardHoganMcIntosh:00}
{\sc Axelsson, A., Grognard, R., Hogan, J., and McIntosh, A.}
\newblock Harmonic analysis of {D}irac operators on {L}ipschitz domains.
\newblock In {\em Clifford analysis and its applications (Prague, 2000)},
  vol.~25 of {\em NATO Sci. Ser. II Math. Phys. Chem.} Kluwer Acad. Publ.,
  Dordrecht, 2001, pp.~231--246.

\bibitem{AxMcIntosh:04}
{\sc Axelsson, A., and McIntosh, A.}
\newblock Hodge decompositions on weakly {L}ipschitz domains.
\newblock In {\em Advances in analysis and geometry}, Trends Math.
  Birkh\"auser, Basel, 2004, pp.~3--29.

\bibitem{BuffaCiarlet1:01}
{\sc Buffa, A., and Ciarlet, P., J.}
\newblock On traces for functional spaces related to {M}axwell's equations.
  {I}. {A}n integration by parts formula in {L}ipschitz polyhedra.
\newblock {\em Math. Methods Appl. Sci. 24}, 1 (2001), 9--30.

\bibitem{BuffaCiarlet2:01}
{\sc Buffa, A., and Ciarlet, P., J.}
\newblock On traces for functional spaces related to {M}axwell's equations.
  {II}. {H}odge decompositions on the boundary of {L}ipschitz polyhedra and
  applications.
\newblock {\em Math. Methods Appl. Sci. 24}, 1 (2001), 31--48.

\bibitem{BuffaCostabelSheen:02}
{\sc Buffa, A., Costabel, M., and Sheen, D.}
\newblock On traces for {$H(curl,\Omega)$} in {L}ipschitz domains.
\newblock {\em J. Math. Anal. Appl. 276}, 2 (2002), 845--867.

\bibitem{CMcM:82}
{\sc Coifman, R.~R., McIntosh, A., and Meyer, Y.}
\newblock L'int\'egrale de {C}auchy d\'efinit un op\'erateur born\'e sur
  {$L\sp{2}$} pour les courbes lipschitziennes.
\newblock {\em Ann. of Math. (2) 116}, 2 (1982), 361--387.

\bibitem{ColtonKress:83}
{\sc Colton, D., and Kress, R.}
\newblock {\em Integral equation methods in scattering theory}, first~ed.
\newblock John Wiley \& Sons, New York, 1983.

\bibitem{Costabel:88}
{\sc Costabel, M.}
\newblock Boundary integral operators on {L}ipschitz domains: elementary
  results.
\newblock {\em SIAM J. Math. Anal. 19}, 3 (1988), 613--626.

\bibitem{EpsteinGreengardONeil:19}
{\sc Epstein, C., Greengard, L., and O'Neil, M.}
\newblock A high-order wideband direct solver for electromagnetic scattering
  from bodies of revolution.
\newblock {\em J. Comput. Phys. 387\/} (2019), 205--229.

\bibitem{FabesJodeitLewis:77}
{\sc Fabes, E., Jodeit, M., and Lewis, J.}
\newblock Double layer potentials for domains with corners and edges.
\newblock {\em Indiana Univ. Math. 26}, 1 (1977), 95--114.

\bibitem{GaneshHawkinsVolkov:14}
{\sc Ganesh, M., Hawkins, S.~C., and Volkov, D.}
\newblock An all-frequency weakly-singular surface integral equation for
  electromagnetism in dielectric media: reformulation and well-posedness
  analysis.
\newblock {\em J. Math. Anal. Appl. 412}, 1 (2014), 277--300.

\bibitem{HaBaMaYo14}
{\sc Hao, S., Barnett, A.~H., Martinsson, P.~G., and Young, P.}
\newblock High-order accurate methods for {N}ystr\"{o}m discretization of
  integral equations on smooth curves in the plane.
\newblock {\em Adv. Comput. Math. 40}, 1 (2014), 245--272.

\bibitem{Hels11}
{\sc Helsing, J.}
\newblock The effective conductivity of arrays of squares: Large random unit
  cells and extreme contrast ratios.
\newblock {\em J. Comput. Phys. 230}, 20 (2011), 7533--7547.

\bibitem{Hels18}
{\sc Helsing, J.}
\newblock Solving integral equations on piecewise smooth boundaries using the
  {RCIP} method: a tutorial.
\newblock {\em arXiv e-prints}, arXiv:1207.6737v9 [physics.comp-ph] (revised 2018).

\bibitem{HelsHols15}
{\sc Helsing, J., and Holst, A.}
\newblock Variants of an explicit kernel-split panel-based {N}ystr\"{o}m
  discretization scheme for {H}elmholtz boundary value problems.
\newblock {\em Adv. Comput. Math. 41}, 3 (2015), 691--708.

\bibitem{HelsingKarlsson:16}
{\sc Helsing, J., and Karlsson, A.}
\newblock Determination of normalized electric eigenfields in microwave
  cavities with sharp edges.
\newblock {\em J. Comput. Phys. 304\/} (2016), 465--486.

\bibitem{HelsingKarlsson:17}
{\sc Helsing, J., and Karlsson, A.}
\newblock Resonances in axially symmetric dielectric objects.
\newblock {\em IEEE Trans. Microw. Theory Tech. 65}, 7 (2017), 2214--2227.

\bibitem{HelsingKarlsson:18}
{\sc Helsing, J., and Karlsson, A.}
\newblock On a {H}elmholtz transmission problem in planar domains with corners.
\newblock {\em J. Comput. Phys. 371\/} (2018), 315--332.

\bibitem{HelsingKarlsson:20}
{\sc Helsing, J., and Karlsson, A.}
\newblock An extended charge-current formulation of the electromagnetic
  transmission problem.
\newblock {\em SIAM J. Appl. Math. 80}, 2 (2020), 951--976.

\bibitem{HelsingKarlssonRosen:20}
{\sc Helsing, J., Karlsson, A., and Ros\'en, A.}
\newblock Comparison of integral equations for the {M}axwell transmission
  problem with general permittivities.
\newblock {\em arXiv e-prints}, arXiv:2007.12260 [physics.comp-ph].

\bibitem{HelsingPerfekt:18}
{\sc Helsing, J., and Perfekt, K.-M.}
\newblock The spectra of harmonic layer potential operators on domains with
  rotationally symmetric conical points.
\newblock {\em J. Math. Pures Appl. 118\/} (2018), 235--287.

\bibitem{HernHerr:19}
{\sc Hern\'andez-Herrera, A.}
\newblock Higher dimensional transmission problems for {D}irac operators on
  {L}ipschitz domains.
\newblock {\em J. Math. Anal. Appl. 478}, 2 (2019), 499--525.

\bibitem{KleiMart88}
{\sc Kleinman, R., and Martin, P.}
\newblock On single integral equations for the transmission problem of
  acoustics.
\newblock {\em SIAM J. Appl. Math. 48}, 2 (1988), 307--325.

\bibitem{KresRoac78}
{\sc Kress, R., and Roach, G.}
\newblock Transmission problems for the {H}elmholtz equation.
\newblock {\em J. Math. Phys. 19}, 6 (1978), 1433--1437.

\bibitem{LaiOneil:19}
{\sc Lai, J., and O'Neil, M.}
\newblock An {FFT}-accelerated direct solver for electromagnetic scattering
  from penetrable axisymmetric objects.
\newblock {\em J. Comput. Phys. 390\/} (2019), 152--174.

\bibitem{MarmMitreaShi:12}
{\sc Marmolejo-Olea, E., Mitrea, I., Mitrea, M., and Shi, Q.}
\newblock Transmission boundary problems for {D}irac operators on {L}ipschitz
  domains and applications to {M}axwell's and {H}elmholtz's equations.
\newblock {\em Trans. Amer. Math. Soc. 364}, 8 (2012), 4369--4424.

\bibitem{McIntoshMitrea:99}
{\sc McIntosh, A., and Mitrea, M.}
\newblock Clifford algebras and {M}axwell's equations in {L}ipschitz domains.
\newblock {\em Math. Methods Appl. Sci. 22}, 18 (1999), 1599--1620.

\bibitem{Muller69}
{\sc M\"{u}ller, C.}
\newblock {\em Foundations of the mathematical theory of electromagnetic
  waves}.
\newblock Revised and enlarged translation from the German. Die Grundlehren der
  mathematischen Wissenschaften, Band 155. Springer-Verlag, Berlin, 1969.

\bibitem{PicardDir:84}
{\sc Picard, R.}
\newblock On the low frequency asymptotics in electromagnetic theory.
\newblock {\em J. Reine Angew. Math. 354\/} (1984), 50--73.

\bibitem{PicardDir:85}
{\sc Picard, R.}
\newblock On a structural observation in generalized electromagnetic theory.
\newblock {\em J. Math. Anal. Appl. 110}, 1 (1985), 247--264.

\bibitem{RosenSpin:17}
{\sc Ros\'en, A.}
\newblock A spin integral equation for electromagnetic and acoustic scattering.
\newblock {\em Appl. Anal. 96}, 13 (2017), 2250--2266.

\bibitem{RosenBoosting:19}
{\sc Ros\'{e}n, A.}
\newblock Boosting the {M}axwell double layer potential using a right spin
  factor.
\newblock {\em Integral Equations Operator Theory 91}, 3 (2019), Paper No. 29,
  25.

\bibitem{RosenGMA:19}
{\sc Ros\'{e}n, A.}
\newblock {\em Geometric multivector analysis}.
\newblock Birkh\"{a}user Advanced Texts: Basler Lehrb\"{u}cher. [Birkh\"{a}user
  Advanced Texts: Basel Textbooks]. Birkh\"{a}user/Springer, Cham, [2019]
  \copyright 2019.
\newblock From Grassmann to Dirac.

\bibitem{SchultzHiptmair:20}
{\sc Schultz, E., and Hiptmair, R.}
\newblock First-kind boundary integral equations for the {D}irac operator in
  {3D} {L}ipschitz domains.
\newblock {\em arXiv e-prints}, arXiv:2012.11994 [math.AP].

\bibitem{TaskinenVanska:07}
{\sc Taskinen, M., and V\"ansk\"a, S.}
\newblock Current and charge integral equation formulations and {P}icard's
  extended {M}axwell system.
\newblock {\em IEEE Trans. Antennas and Propagation 55}, 12 (2007), 3495--3503.

\bibitem{TaskinenYlaOijala:06}
{\sc Taskinen, M., and Yl\"a-Oijala, P.}
\newblock Current and charge integral equation formulation.
\newblock {\em IEEE Trans. Antennas and Propagation 54}, 1 (2006), 58--67.

\bibitem{VicGreFer18}
{\sc Vico, F., Greengard, L., and Ferrando, M.}
\newblock Decoupled field integral equations for electromagnetic scattering
  from homogeneous penetrable obstacles.
\newblock {\em Comm. Part. Differ. Equat. 43}, 2 (2018), 159--184.

\bibitem{Weck:04}
{\sc Weck, N.}
\newblock Traces of differential forms on {L}ipschitz boundaries.
\newblock {\em Analysis (Munich) 24}, 2 (2004), 147--169.

\end{thebibliography}

\end{document}